\documentclass[10pt,english,a4paper]{smfart}
\usepackage[margin=2.8cm,footskip=1cm]{geometry}
 
\usepackage{amsfonts}
\usepackage{phoenician} %don't worry 'bout it
\usepackage{runic}
\usepackage{amsmath}
\usepackage{amsthm}
\usepackage{dsfont}
\usepackage{comment}
\usepackage{amssymb}
\usepackage{xcolor}
\usepackage{hyperref}
\usepackage{breqn}
\usepackage{subcaption}
\usepackage{enumitem}
\usepackage{stmaryrd}
\usepackage{graphicx}
\usepackage[
backend=biber,
style=alphabetic,
sorting=ynt
]{biblatex}
\usepackage{bigints}
\newtheorem{theorem}{Theorem}[section]
\newtheorem{corollary}{Corollary}[section]
\newtheorem{proposition}{Proposition}[section]
\newtheorem{lemma}{Lemma}[theorem]
\theoremstyle{definition}
\newtheorem{definition}{Definition}[section]
\theoremstyle{remark}
\newtheorem{remark}{Remark}

\newcommand{\R}{\mathbb{R}}

\newcommand{\N}{\mathbb{N}}

\newcommand{\rs}{\textbf{RS}}
\newcommand{\rb}{(\alpha_{sc}, \beta_{sc})}
\newcommand{\Leb}{\mathrm{Leb}}

\newcommand{\landsc}{\mathrm{L}}
\newcommand{\slope}{\mathrm{t}}
\newcommand{\Half}{\R \times \R_+}
\newcommand{\below}[3]{\mathrm{below}\left(#1, #2, #3\right)}
\newcommand{\belowa}[2]{\mathrm{below}\left(#1, #2\right)}
\newcommand{\primi}{\mathfrak{F}}
\newcommand{\event}{\Xi}
\newcommand{\graph}{\mathcal{G}}
\newcommand{\zone}{\mathrm{Z}}
\newcommand{\placehold}{\mathrm{w}}

\newcommand{\p}{\mathbb{P}}
\newcommand{\e}{\mathbb{E}}
\renewcommand{\epsilon}{\varepsilon}
\renewcommand{\max}{\text{max}}

\author{Paul Rax}
\address{MathNet Team, INRIA Paris}
\email{paul-pierre.rax@inria.fr}
\urladdr{\href{https://paul-rax.github.io/webpage/}{paul-rax.github.io/webpage}}

\author{François Baccelli}
\address{DI ENS \& MathNet Team, INRIA Paris}
\email{Francois.Baccelli@ens.fr}
\urladdr{}

\title[The Relay Random Tree]{The Relay Random Tree:\\A stochastic geometry approach of multihop relay in an urban visibility setting.}

\begin{document}

\maketitle

\begin{abstract}
    In a recent work (\cite{Junse_Francois}), a one dimensional stochastic geometry model was introduced to study Line of Sight (LoS) connections using Reconfigurable Intelligent Surfaces (RIS), in the context of non terrestrial networks. In this model, signal can be propagated in a urban environment, with buildings acting as obstacles with RIS (which, for the scope of this present article can essentially be thought of as relays) on their rooftops, relaying the connection. The present paper extends this model by both allowing arbitrary distributions for the buildings heights, and considering multi-hop connections. Those generalities also lead to considering structural problems linked to the total load of a relay. Furthermore, studying this Line of Sight connection geometry at the light of geometric random graph theory, we show that it constitutes a computationally well understood example that highlights the different classes of the Eternal Family Trees (EFTs) classification (\cite{EFT}).
\end{abstract}
\tableofcontents
\section{Introduction}
Over the recent years, satellite networks and their geometric modeling have been the object of a growing interest (starting from \cite{okati_stochastic_2020}). However, if most models concern the \textit{long range} geometry of the problem, that is focusing on the geometry of the satellites (such as the spherical cap seen by a user), the environment near a terrestrial user can be of crucial interest, as obstacles near a user can impact the connection to a satellite. Using stochastic geometry tools, Baccelli and Lee (\cite{Junse_Francois}) recently proposed a model of blockage, and studied the possible uses of Reconfigurable Intelligent Surfaces (RIS, \cite{di_renzo_reconfigurable_2020}) to cope with them within this framework. We propose here to continue the study of this model, extending some of its results to more general distributions, and to questions linked with the dimensioning of the network. While doing so, we also link this model to the known classification of covariant vertex shifts (\cite{EFT}, \cite{CVPS}), as it constitute an example where this qualitative behaviour can be complemented with exact computations.
Moreover, while the main motivation stems from satellite networks, and the potential uses of RIS, our work enters in the scope of the study of line of sight and obstacle models in telecommunication (\cite{INFOCOM}).

Let us now give a brief overview of the model and results of \cite{Junse_Francois}.

In that model, a user on the line $\R$ sees building as vertical segments on the same line. A blocking building is seen as the one that blocks its line of sight upward, that is the one with the top of which the user forms the largest angle (the so called \textit{visibility angle}). Considering that there are RISs on top of the buildings, the authors study how much one can decrease this angle through the use of either a transmissive or a reflexive process.

Formally, a user $o$ is placed at the origin of the line $\R$. On the rest of the line, a PPP of intensity $\lambda >0$, endowed with i.i.d marks following a Weibull distribution with shape parameter $k$ and scale parameter $\frac{1}{\mu}$ is generated, to model buildings through their position (point process) and height (marks). In this setting, they consider the \textit{right blockage angle} $\Theta^+$ as the greatest angle formed by $o$, the line $\R$, and the top of one of the buildings on the right of $o$. This corresponds to the arc tangent of 
\[\underset{(x, h)\text{ building}}{\max} \frac{h}{x}.\]
The \textit{left blockage angle} is defined similarly. Under the hypotheses on the buildings and their heights, this is always well defined, and so are the corresponding \textit{blocking buildings}, as the unique buildings $(X_+, H_+)$ and $(X_-, H_-)$ achieving the said maxima.\\

\begin{figure}[h]\label{Situ_Junse_Francois}
            \centering
            \includegraphics[width=0.75\textwidth]{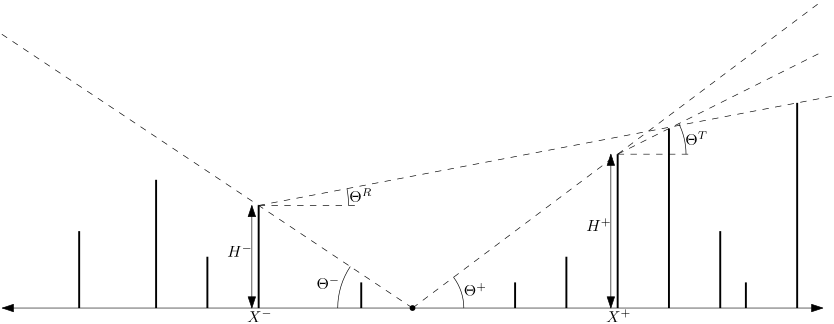}
            \caption{The situation considered in \cite{Junse_Francois}.}
\end{figure}

From this situation, two angles are then defined :
\begin{itemize}
    \item The \textit{transmissive blockage angle} $\Theta^T$, which is the right blockage angle seen from $(X^+, H^+)$, or in other words the enhanced blockage angle obtained through the use of transmissive $RIS$ on the right blocking building.
    \item The \textit{reflexive blockage angle} $\Theta^R$, which is right blockage angle seen from $(X^-, H^-)$, that is the enhanced blockage angle given by the use of reflexive RIS on the top of the left blocking building. 
\end{itemize}

The goal of the present paper is to further study this model in the transmissive case, generalising on the one hand the framework, allowing any kind of height distribution, and looking at different types of relays, and on the other hand extending the questions, by looking at multihop relaying and studying the geometry it induces.

\subsection*{Contributions}

Throughout this article, we extend the results and setting of the transmissive case study of \cite{Junse_Francois}:
\begin{itemize}
    \item In this context, we study the multihop case, giving distribution of the evolution of the successive position and height of the blockage buildings as well as blockage angles. Our methods also extend the model to accommodate any height distribution as well as in-homogeneous Poisson point process for the positions of the buildings.
    \item Still in the above context, we also give precise results on the distribution of the zone of users and buildings that eventually connect to a given building.
    \item We extend our study to the case where the transmissions are limited in range, where we manage to extend our methods and results to some specific ranges (where only the horizontal distance between the building matters, which can be seen as an approximation for dense urban settings), but where they ultimately fail in the general case.
    \item In the last case however we manage to find the distribution of the first blocking building of any given point.
\end{itemize}

Doing so, we highlight the structural properties of the multihop process: we find that, except for the last, general finite range, case, the Poisson point process structure of the buildings yields a Markovian evolution of blockage.

Moreover, in the stationary case, we study the underlying unimodular random tree, the Relay EFT, that gives the structure of the connections in this scheme. Thanks to our analysis, it constitutes a computationally well understood example of unimodular random tree emerging from a point-map. This unimodularity is central in our study as leveraging it is one of the ways we have to recover results in finite range.
\subsection*{Structure of the paper}

First, in Section \ref{model}, we introduce our framework, some basic properties, as well as the quantities of interest in our study. Then, in Section \ref{multihop}, we study the behaviour of a multihop infinite range scheme, that is the direct generalisation of the study of the transmissive enhancement of \cite{Junse_Francois} to the case where multiple successive enhancements are allowed. We give the explicit joint law of the evolution of position and height of the successive blocking buildings, as well as blockage angle, showing that it exhibits a Markovian behaviour. We also briefly prove the intuitive almost-sure convergence of the height, position and angle. In Section \ref{tree}, we then study the geometry generated by such a scenario, using it to answer load questions related to the number of users being relayed by a given RIS, by giving explicit computations and characterisations in this context. Finally, in Section \ref{finite_range}, we extend our methods to the finite range case, where relays are limited in their range.

\subsection*{Acknowledgements.}

The question of multi-hop communications in the context of the model discussed here was raised by Dr. Junse Lee of Sungshin University in Seoul. The authors are grateful for this question of his which is at the origin of the present paper.

\section{Model: definition and first results}\label{model}

Let us first define our model, as well as the notation and vocabulary which we use for the rest of the article.\\
First, we introduce the geometric objects we are interested in when the buildings are seen as marked points on the line.

\subsection{Geometry of the problem}

\begin{definition}\textbf{Landscape, buildings.}

    Let us consider the half plane $\R \times \R_+$. Unless specified otherwise, the term \textit{point} will refer to an element of $\Half$.\\
    A \textit{landscape} $\landsc$ is a set of points $\landsc = {(x_i, h_i)}$ of $\R \times \R_+$, where the set $\{x_i\}$ is locally finite, and which we call \textit{buildings}. We interpret the building $(x, h)\in \landsc$ as the vertical segment from $(x, 0)$ to $(x, h)$. For $(x, y)\in \landsc$ a building, we call $x$ its base, and $y$ its height.\\
\end{definition}
In this setting, the crucial element is \textit{visibility}: we say that a building $(x, h)\in \landsc$ is visible from a point $(a, b)$ if there is no other building of $\landsc$ above the segment $[(a, b),(x, y)]$. This definition changes in the finite range case, but in the current, infinite-range, case, it justifies the definition of the \textit{blocking building}, the furthest visible building, as the locus of a global maximum:
\begin{definition}\textbf{Blockage.}

    For $\landsc$ a landscape and $(x, y)\in \Half$ a point, we call (first) blocking building of $(x, y)$ in $\landsc$ the following quantity when it is well defined:
    \[\tau(x, y, \landsc) = (x_b, y_b) := \underset{\substack{(a, b) \in \landsc \\a>x}}{\text{argmax}} \frac{b-y}{a-x}.\]
    When the blocking building of $(x, y)$ is well defined, we further introduce its blockage angle as:
    \[\theta(x, y, \landsc): = \arctan\left(\frac{y_b - y}{x_b-x}\right).\]
    In the following, when the involved landscape is unequivocal, we will simply write $\tau(x, y)$.
\end{definition}

\begin{remark}
Rather than the blockage angle, we are more interested in its tangent, which we call the \textit{blockage slope}:
\[\slope(x, y, \landsc): = \frac{y_b - y}{x_b-x}.\]
\end{remark}
\begin{remark}
In the following, we call \textit{angle} between (or angle formed by) two points of the half plane $\Half$ the angle between the horizontal line $\R \times {0}$ and the line induced by these two points. Moreover, we endow $\Half$ with the lexicographical order: when comparing points, we in fact compare their $x-$coordinate, and then their $y-$coordinate.
\end{remark}

An important viewpoint of this situation consists in formulating the problem in terms of \textit{shade}:

\begin{definition}\textbf{Shade.}

Let $\landsc$ be a landscape, $(x_1, y_1)$ a point of $\Half$, and $(a, b)\in \landsc$ a building. We call the \textit{shade} of $(a, b)$ from $(x_1, y_1)$ the set of points right to $a$ that form a lower angle with $(x_1, y_1)$ than $(a, b)$ does:
\[S^{(x_1, y_1)}_{(a, b)} = \left\{ (x, y)\in \Half:x \geq a,\  \frac{y-y_1}{x-x_1}\leq \frac{b-y_1}{a-x_1}\right\}.\]
\end{definition}

\begin{figure}[h]\label{shade_fig}
            \centering
            \includegraphics[width=0.75\textwidth]{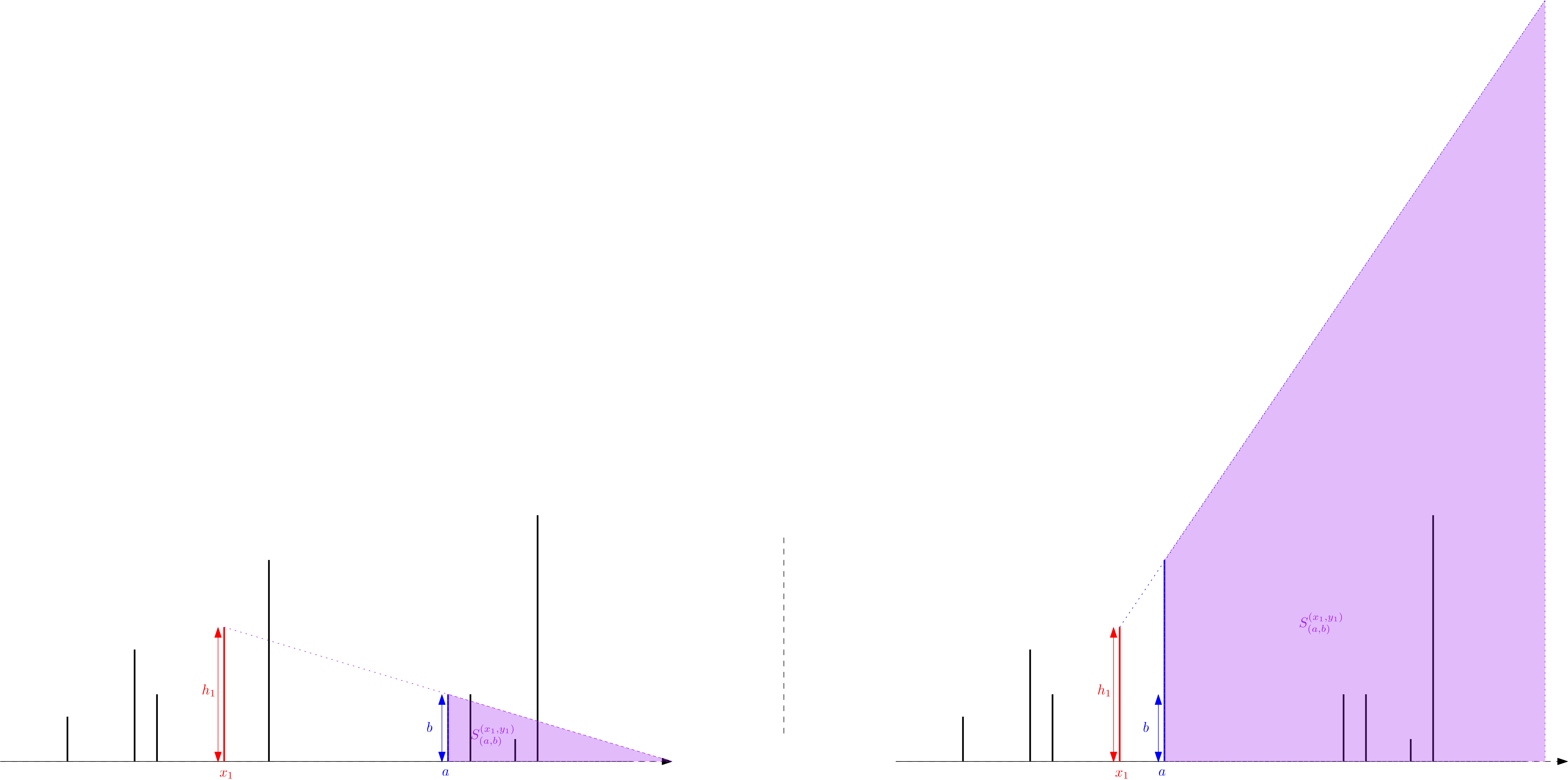}
            \caption{The shade can be seen as the part of the half plane that is in the shadow of $(a, b)$ from a light source situated in $(x_1, h_1)$.}
\end{figure}
        
This notion is the natural one for the model described here. Indeed, the shade of a building from a point corresponds to the locations of $\Half$ of which this building blocks direct transmission from the said point. It also allows us to reformulate the definition of blocking building: a building $(a_1, b_1)$ is the blocking building of a point $(x, y)$ if and only if the shade $S^{(x, y)}_{(a_1, b_1)}$ contains that of all subsequent buildings, i.e.
\[\forall (a, b)\in \landsc \ \mbox{st}\ a\geq a_1,\ S^{(x, y)}_{(a, b)}\subset S^{(x, y)}_{(a_1, b_1)}.\]
This is equivalent to requiring that $S^{(x, y)}_{(a_1, b_1)}$ contains all the buildings on the right of $a_1$:
\[\forall (a, b)\in \landsc \ \mbox{st}\ a\geq a_1,\ (a, b)\in S^{(x, y)}_{(a_1, b_1)}.\]

Moreover, this notion of \textit{shade} is the main idea behind the characterisation of \textit{eventual relay} through the notion of \textit{reverse shade}, central in Section \ref{tree}, and the finite range case clearly highlights the differences and limitations we face in Section \ref{finite_range}.

Since the geometric aspect of the problem boils down to understanding the angles formed by tops of buildings, let us now state a basic fact on the angles formed by points of the half plane that come into play in the following sections:

\begin{proposition}\label{geomsimple}
Let $(x_1, y_1)$, $(x_2, y_2)$ and $(x_3, y_3)$ be three points of the half plane $\Half$ such that $x_1\leq x_2 \leq x_3$. Then the angle formed by $(x_1, y_1)$ and $(x_3, y_3)$ is between that formed by $(x_1, y_1)$ and $(x_2, y_2)$ and that formed by $(x_2, y_2)$ and $(x_3, y_3)$.
\end{proposition}
\begin{proof}
This is simply because the tangent of the angle formed by $(x_1, y_1)$ and $(x_3, y_3)$ is a convex combination of the two others:
\[\frac{y_3 - y_1}{x_3-x_1} = \frac{y_2-y_1}{x_2-x_1}\frac{x_2-x_1}{x_3-x_1}+\frac{y_3-y_2}{x_3-x_2}\frac{x_3-x_2}{x_3-x_1}.\]
\end{proof}

This gives rise to an important property for the infinite range case: 

\begin{corollary}\label{increase}\textbf{Increase of information.}\\
    Let $(x, y)$ be a point of $\Half$ and $\landsc$ be a landscape of buildings. Then the angle formed by $(x, y)$ and its blocking building $\tau(x, y)$ is greater than that formed by $\tau(x, y)$ and its own blocking building $\tau (\tau (x, y))$. In other words, the sequence of shades $\left(S^{\tau^{(n)}(x, y)}_{\tau^{(n+1)}(x, y)}\right)_{n\in \N}$ is non increasing with respect to inclusion.
\end{corollary}

\begin{figure}[h]
            \centering
            \includegraphics[width=0.75\textwidth]{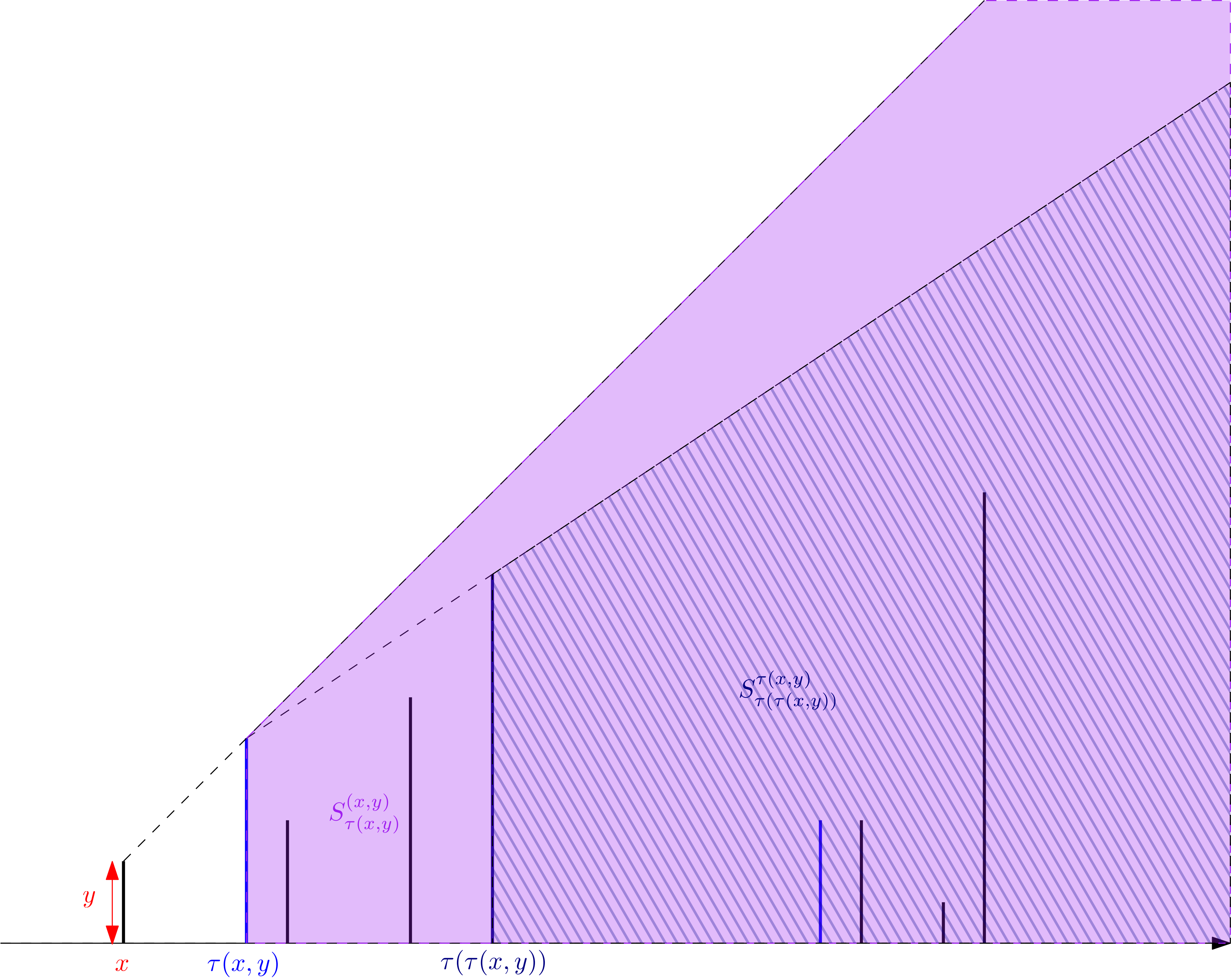}
            \caption{\label{increase_figure}The convexity of the shades implies this non-increasing character of the shades: Indeed, since the blocking building $\tau(x, y)$ as well as all buildings to its right are in the shade $S^{(x, y)}_{\tau(x, y)}$, the shade $S^{\tau(x, y)}_{\tau(\tau(x, y))}$ is included in $S^{(x, y)}_{\tau(x, y)}$.}
\end{figure}

Having defined the quantities of interest for a fixed landscape, let us now look at the types of landscapes we consider:

\subsection{Random setting}

In the following, as in \cite{Junse_Francois}, we assume the buildings are distributed as a marked Poisson point process with i.i.d. marks. However, we assume more general forms for the distributions of the heights (i.e., the marks). \\

In the following, for $\lambda > 0$ and $\mathcal{L}$ a probability measure on $\R_+$, we call a Poisson building process with intensity $\lambda$ and height distribution $\mathcal{L}$ a marked Poisson Point Process on $\R$ with i.i.d. marks (independent from the process itself) having law $\mathcal{L}$.

In this paper, we call such landscapes Poisson building processes. For our study, we are interested in the height distributions that make it possible to define the blocking building and blockage angle. Let $L$ be a Poisson building process with intensity $\lambda >0$ and height distribution $\mathcal{L}$.\\

One can notice that having $\underset{(a, b) \in L, a>0}{\text{max}} \frac{b}{a}$ almost surely well defined is equivalent to having $\mathcal{L}$ integrable, which is in turn equivalent to having $\underset{(a, b) \in L, a>0}{\text{max}} \frac{b-y}{a-x}$ almost surely well defined for all $(x, y)\in \Half$ (this comes from the law of large numbers, but one can see it in terms of thinning for marked PPP, see for instance Proposition \ref{dessous_diago}). As such, in the following we always assume that $\mathcal{L}$ has finite mean.

However, this is not enough to ensure that the blocking building is well defined, as the said maximum needs to be reached in a unique building for our current framework. A necessary and sufficient condition for that is to have $\mathcal{L}$ not put any positive mass on the supremum of its support. Indeed, the only lines having a positive probability to have the top of building on them are the horizontal lines with the height being an atom of $\mathcal{L}$. This, combined with the Markov property as well as the independence of marks ensures that, for any point of the half plane, it is aligned with two (and in practice an infinite amount) tops of buildings if and only if its height $h$ is that of an atom of $\mathcal{L}$, in which case the tops it is aligned with are also at said height, and as such, the angle they form is $0$. But this can be the blockage angle if and only if there are no taller buildings, i.e., $h$ is the supremum of the support of $\mathcal{L}$. This specific case, where the blockage angle eventually reaches $0$, once reformulated as to be part of our framework, is of particular interest in Section \ref{tree}, as it exhibits a different qualitative behaviour, linked with a difference in terms of the classification of covariant vertex shifts developed in \cite{EFT}.\\

Thanks to these observations, we are now ready to formulate our results on this framework.

\section{Evolution in the multihop case}\label{multihop}

In this section, we assume the following:
\begin{itemize}
    \item $\mathcal{L}$ is a law on $\R_+$ that has finite means and puts no mass on the maximum of its support, and $\lambda >0$ is a positive real number.
    \item $L$ is a random landscape that is a Poisson building process of intensity $\lambda$ and height distribution $\mathcal{L}$.
\end{itemize}

The goal of the section is to understand the law of the sequence $(\tau^{(n)}(0, 0, L))_{n\in \N}$ of successive blocking buildings of a typical user placed in this landscape. However, for the sake of completeness, as well as for highlighting some later properties of this evolution, we assume $(x_0, h_0)$ to be a random point in $\Half$, distributed according to a certain law $\mu$. From this, let us introduce some notation:
\begin{itemize}
    \item We let $F$ be the CDF of $\mathcal{L}$, and $\primi$ be the primitive of $1-F$ that goes to zero at $+\infty$ (it exists since $\mathcal{L}$ is integrable).
    \item We let the sequence $(X_n, H_n)_{n\in\N}$ be the blockage sequence $(\tau^{(n)}(x_0, h_0, L))_{n\in \N}$, and for $n\geq 1$, $\mathrm{t}_n$ be the tangent of the $n$-th blockage angle: 
    \[\mathrm{t}_n = \frac{H_n-H_{n-1}}{X_n - X_{n-1}}.\]
\end{itemize}

\begin{figure}[h]\label{Figure_seq}
            \centering
            \includegraphics[width=0.75\textwidth]{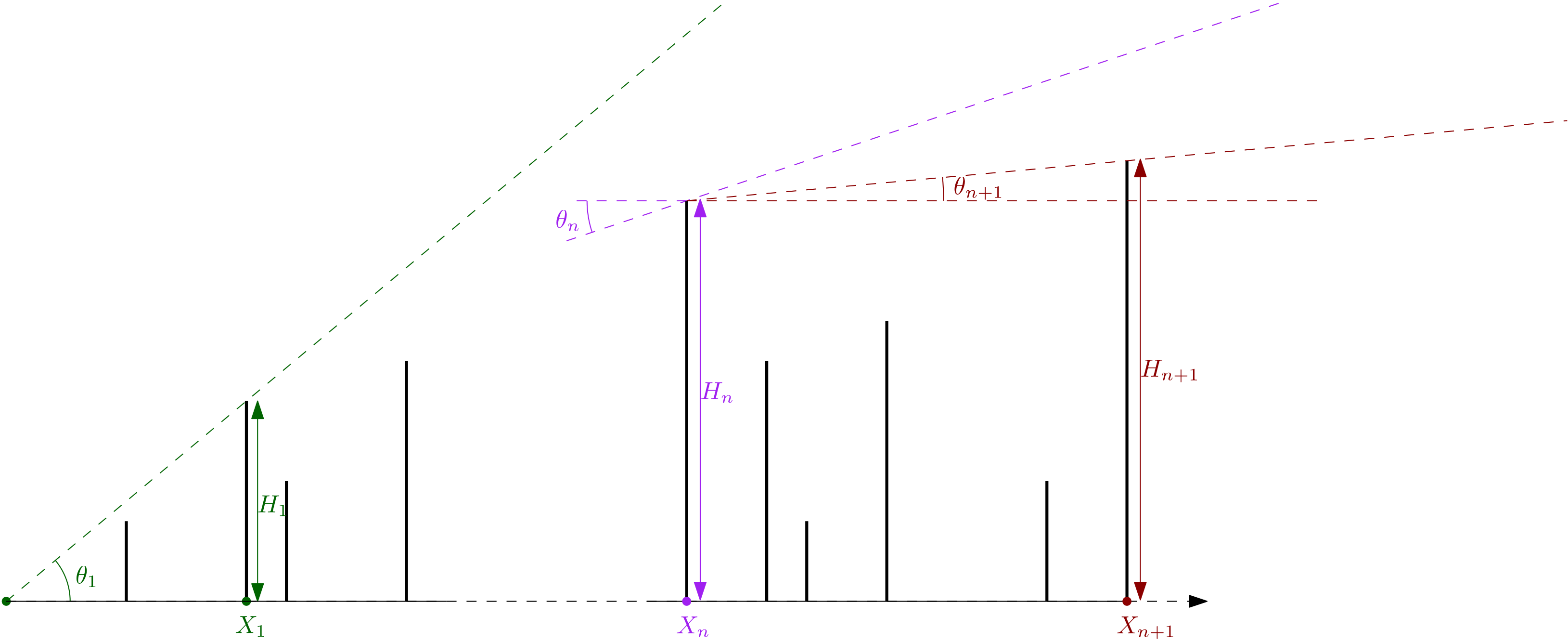}
            \caption{The blockage sequence seen from the origin.}
\end{figure}
The main result of this section is the following:
\begin{theorem}\label{evo}\textbf{Evolution of blockage.}

Let $N$ be a fixed integer.
The sequence $(X_n, H_n)_{0\leq n\leq N}$ has density $g_N$ with respect to the product measure $\mu\displaystyle\bigotimes_{i=1}^{N} (\Leb \otimes \mathcal{L})$ (we consider here the Lebesgue measure on $\R_+$), which is given by
\[g_N: (x_n, h_n)_{0\leq n\leq N} \mapsto \lambda^N\exp\left(-\lambda\left(\frac{-\primi(h_0)}{t_1}+\sum_{i=1}^{n-1}\primi(h_i)(\frac{1}{t_{i}}-\frac{1}{t_{i+1}})\right)\right)\prod_{1\leq i\leq n}\mathds{1}_{h_i\geq h_{i-1}}\mathds{1}_{t_i\geq t_{i+1}},\]
where $t_i$ is defined to be the slope from $(x_{i-1}, h_{i-1})$ to $(x_i, h_i)$.
\end{theorem}

This density can be decomposed in a product of terms containing only quantities depending of the $i$-th and $(i+1)$-th jump, hinting at an evolution with a loss of memory. A simple change of variable yields the following characterisation:

\begin{corollary}\label{Markov}\textbf{Markovian characteristic of the evolution of blockage.}\\
    The sequence $(\mathrm{t}_n, H_n)$ is a Markov chain with its kernel $\mathrm{VIS}$ given by:
    \begin{align*}
        &\forall (T, H)\in \Half, \forall A \in \mathcal{B}(\Half), \mathrm{VIS}((T, H), A) = \int_A \mbox{vis}^{T, H}(t, h) \rm{d}\Leb(t) \rm{d}\mathcal{L}(h)\\
        &\forall T, H, t, h \in \R_+, \mbox{vis}^{T, H} (t, h) :=\mathds{1}_{h> H, 0\leq t\leq  T} \frac{h-H}{t^2}\lambda \exp\left(-\lambda \primi(H)(\frac{1}{T}-\frac{1}{t})\right).
    \end{align*}
\end{corollary}

\begin{remark}
As a sanity check, one can verify that, by setting $N = 1$ or $N=2$ and choosing $\mathcal{L}$ to be exponential, our results correspond to those of \cite{Junse_Francois} in the transmissive case.
\end{remark}

\begin{remark}
The Markov property can be understood from Corollary \ref{increase}: the building process to the right of a given blocking building $(X_n, H_n)$ is, from the Markov property of the independently marked Poisson Point Process, a Poisson building process conditioned on stay in the successive shades $\left(S^{(X_i,H_i)}_{(X_{i-1}, H_{i-1})}\right)$. But the decreasing characteristic of the shades makes it that all this conditioning is captured by the condition of having all the buildings to the right in the last shade $S^{(X_n,H_n)}_{(X_{n-1}, H_{n-1})}$ (see fig. \ref{increase_figure}).
\end{remark}

One can also analyse the law in Theorem \ref{evo} as the following:
\begin{itemize}
    \item The measure $\lambda^N\displaystyle\bigotimes_{i=1}^{N} (\Leb \otimes \mathcal{L})$ represents the presence of buildings in $(x_n, h_n)_{0\leq n\leq N}$.
    \item The indicators come from the \textit{increase of information}.
    \item The last exponential term represents the successive conditioning of the Poisson building process we mentioned above.
\end{itemize}
As such, before proving the theorem, let us first recall a classical thinning fact (see for instance \cite{baccelli_random_2024}) for PPP, that is central in getting the exponential term.

Let $\Phi$ be a point process, and $A$ be a measurable portion of the line $\R$ and $f:\R\rightarrow \R_+$ a function. We denote \[\belowa{A}{f}:=\{\forall i \in \N\text{ st }X_i\in A\cap\Phi, H_i\leq f(X_i)\}\] the event that all buildings located in $A$ have their heights lower than the curve of the function $f$.

\begin{lemma}\label{dessous_diago}
    Let $(X_i)$ be a Poisson Point Process of intensity $\lambda'$, marked with heights $(H_i)$ being i.i.d random variables of law $\mathcal{L}$. Then, for any given function $f$ that is Lebesgue measurable and positive on $A$, the probability of $\belowa{A}{f}$ is:
    \[\exp \left(\int_A -\lambda' \mathcal{L}(]f(x), \infty[) \rm{d}x\right) .\]
\end{lemma}

\begin{proof}
We consider the point process $\Phi\cap A$  on $A$. The event $\{\forall i \in \N, H_i\leq f(X_i)\}$ can be reformulated as $\{\{X_i : H_i> f(X_i)\} = \emptyset\}$. But the set of points $\{X_i : H_i> f(X_i)\}$ is a thinning of $(X_i)$ with retention function 
\[p:x \mapsto \mathcal{L}(]f(x), +\infty[).\]
As such, $\{X_i : H_i> f(X_i)\}$ is a PPP of intensity $\Lambda: B \mapsto \int_B \lambda'p(x)\rm{d}x$, which concludes our proof.
\end{proof}

In this paper, we mainly use this result in the case the function $f$ is linear, and as such we denote
\[\below{A}{(a, b)}{t}:=\belowa{A}{f: x\mapsto t(x-a)+b}\] 
the event that all buildings located in $A$ have their heights lower than the line of slope $t$ and passing through the point $(a, b)$.\\
Let us now prove Theorem \ref{evo}:

\begin{proof}
Let $N$ be fixed. Our objective is to compute the probability of $(X_n, H_n)_{0\leq n\leq N}$ being in a product of intervals.\\
To do so, let us consider the event 
\[\event  = \{\forall 0\leq i\leq N, X_i\in [a_i, a_i+\epsilon_i] \text{ and } H_i\in [b_i, b_i+\eta_i]\}.\]
For the sequences $(a_i)$, $(b_i)$, $(\epsilon_i)$ and $(\eta_i)$ given.\\
First, we bound downward $\event$. To do so, let us decompose the event $\event$: it is the conjunction of the following events:
\begin{enumerate}
    \item There are buildings $(x_i, h_i)$ in each of the $[a_i, a_i+\epsilon_i]\times[b_i, b_i+\eta_i]$.
    \item $\displaystyle\bigcap_{i=1}^{N}\below{[x_{i-1}, +\infty[}{(x_{i-1}, h_{i-1})}{t_i}$, where $t_i := \frac{h_{i+1}-h_i}{x_{i+1}-x_i}$.
\end{enumerate}
By Proposition \ref{geomsimple}, and the almost sure decreasing characteristic of the blockage angles, one can reformulate the second point to get the following characterisation:
\begin{enumerate}
    \item There are buildings $(x_i, h_i)$ in each of the $[a_i, a_i+\epsilon_i]\times[b_i, b_i+\eta_i]$.
    \item $\displaystyle\bigcap_{i=1}^{N}\below{[x_{i-1}, x_{i}]}{(x_{i-1}, h_{i-1})}{t_i}$, where $t_i := \frac{h_i-h_{i-1}}{x_i-x_{i-1}}$.
    \item $\below{[x_N, +\infty[}{(x_N, h_N)}{t_N}$.
    \item $(t_i)$ is decreasing and $(h_i)$ is increasing.
\end{enumerate}
We now want to find lower and higher bounds for each of these events that are independent from each other. We make use of the spatial Markov property of Poisson Point Processes to change the first three events in events taking place on disjoint domains of the line $\R$: the first only involves points with $x-$coordinate in $[a_i, a_i+\epsilon_i]$ and the second only involves buildings which are between $a_i+\epsilon_i$ and $a_{i+1}$. For the fourth, we simply impose conditions on the $a_i$, $b_i$, $\epsilon_i$ and $\eta_i$.

To that end, for $1\leq i\leq N$, we define
    \begin{align*}
    & t^{min}_{i}:= \frac{b_{i}-b_{i-1}-\eta_{i-1}}{a_{i}+\epsilon_i-a_{i-1}}.
    \end{align*}
    The variable $t^{min}_{i}$ can be interpreted as the the minimal slope formed by points in $[a_i, a_i + \epsilon_i]\times [b_i, b_i+\eta_i]$ and $[a_{i+1}, a_{i+1} + \epsilon_{i+1}]\times [b_{i+1}, b_{i+1}+\eta_{i+1}]$ (see fig.\ref{t_min}).

\begin{figure}[h]
            \centering
            \includegraphics[width=0.3\textwidth]{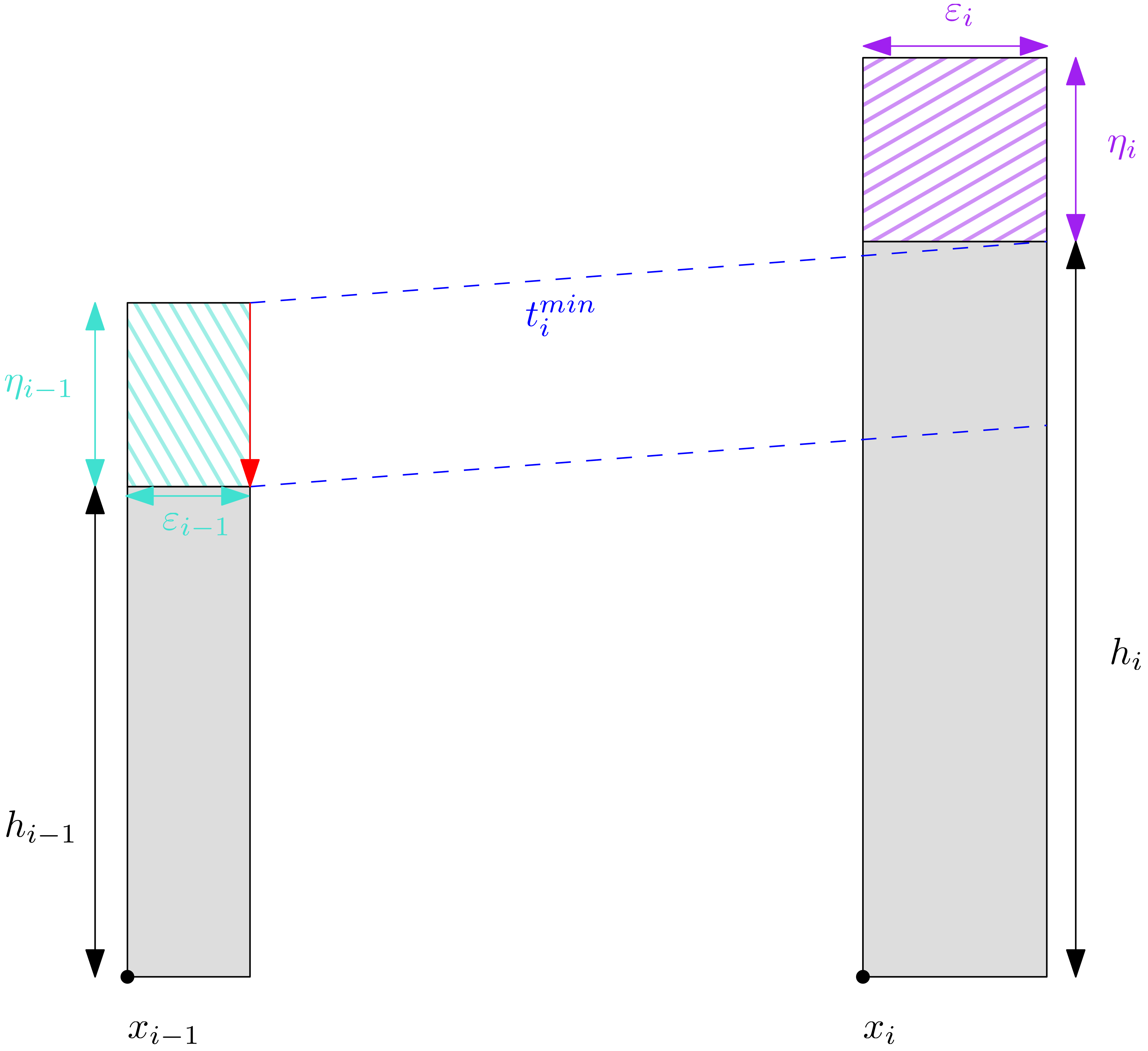}
            \caption{\label{t_min}Asking for having two points in the dashed areas such that every building between $x_{i-1}+\epsilon_{i-1}$ and $x_i$ is below the line between the two is weaker than asking that all buildings between the two wanted $x-$coordinates are under the line of slope $t^{min}_i$ and starting from $(x_{i-1} +\epsilon_{i-1}, h_{i-1})$.}
\end{figure}

Now, thanks to these notations and observations, we can estimate the probability of $\event$:\\
\textbf{Lower bound on $\p(\event)$.}

We have the following conjunction that implies $\event$:

\begin{enumerate}
    \item There is exactly one building in each of the $[a_i, a_i+\epsilon_i]$, and its height is in $[b_i, b_i +\eta_i]$.
    \item $\displaystyle\bigcap_{i=1}^{N}\below{[a_{i-1}+\epsilon_{i-1}, a_i]}{(a_{i-1}+\epsilon_{i-1}, b_{i-1})}{t^{min}_i}$.
    \item $\below{[a_N+\epsilon_N, +\infty[}{(a_N+\epsilon_N, b_N+\eta_N)}{t^{min}_N}$.
    \item $\forall 1\leq i\leq N, \ b_i\geq b_{i-1}+\epsilon_{i-1}$ and $t^{max}_i \leq t^{min}_{i-1}$.
\end{enumerate}

As such, let us denote 
\begin{align*}
    \event^{-} := &\left( \bigcap_{1\leq i \leq N} \{\text{there is exactly one building }(x, h)\text{ st }x \in [a_i, a_i+\epsilon_i],\text{ and it satisfies } h\in [b_i, b_i +\eta_i]\}\right)\\
    &\bigcap \left(\bigcap_{0\leq i \leq N-1} \below{[a_i+\epsilon_i, a_{i+1}]}{(a_i+\epsilon_i}{b_i),t^{min}_{i+1}}\right)\\
&\bigcap \below{[a_N+\epsilon_N, \infty[}{(a_N+\epsilon_N, b_N)}{t^{min}_N}\\
&\bigcap\left(\{\forall 1\leq i\leq N, \ b_i\geq b_{i-1}+\epsilon_{i-1}\}\cap \{\forall 1\leq i\leq N, \ t^{max}_i \leq t^{min}_{i-1}\}\right).
\end{align*}

All the events in this intersection are independent from the domain Markov property and the independence of the marks, so it suffices to compute their probabilities.

On the one hand, the probability of the last one is quite straightforward to compute, as we have:

\[\p\Biggl(\{\forall 1\leq i\leq N, \ b_i\geq b_{i-1}+\epsilon_{i-1}\}\cap \{\forall 1\leq i\leq N, \ t^{max}_i \leq t^{min}_{i-1}\Biggl) = \prod_{i=1}^{N}\mathds{1}_{b_i\geq b_{i-1}+\epsilon_{i-1}}\mathds{1}_{t^{max}_i \leq t^{min}_{i-1}}.\]

Then, the probability of the first event is
\begin{align*}
&\p\left( \bigcap_{1\leq i \leq N} \{\text{there is exactly one building }(x, h)\text{ st }x \in [a_i, a_i+\epsilon_i],\text{ and it satisfies } h\in [b_i, b_i +\eta_i]\}\right)\\
&= \mu([a_0, a_0+\epsilon_0]\times [\mathrm{b}_0, \mathrm{b}_0+\eta_0])\prod_{i = 1}^{N} \lambda\epsilon_i \exp(-\epsilon_i) \mathcal{L}([b_i, b_i +\eta_i]).
\end{align*}

For the second event, we can use Lemma \ref{dessous_diago} to get:
\begin{align*}
    \p\Biggl(&\below{[a_i+\epsilon_i, a_{i+1}]}{(a_i+\epsilon_i}{t^{min}_{i+1}}\Biggl) \\
    &= \exp \left(\displaystyle\int_{a_i+\epsilon_i}^{a_{i+1}} -\lambda \mathcal{L}(]t^{min}_{i+1} (x-a_i-\epsilon_i)+b_i, \infty[) \rm{d}x\right)\\
    &=  \exp \left( -\frac{\lambda}{t^{min}_{i+1}} ( \primi(b_i + (b_{i+1} - b_i -\eta_i)\frac{a_{i+1}-a_i-\epsilon_i}{a_{i+1}+\epsilon_{i+1}-a_i})-\primi (b_i))\right)\\
    &= \exp \left( -\frac{\lambda}{t^{min}_{i+1}} ( \primi(b_{i+1} )-\primi (b_i))\right)\exp \left( -\frac{\lambda}{t^{min}_{i+1}} ( \primi(b_i + (b_{i+1} - b_i -\eta_i)\frac{a_{i+1}-a_i-\epsilon_i}{a_{i+1}+\epsilon_{i+1}-a_i})-\primi (b_{i+1}))\right).
\end{align*}
The last exponential is greater than $1$ since $\primi$ is increasing.
And lastly, from our choice of $\primi$, we have:
\begin{align*}
    \p\Biggl(\below{[a_N+\epsilon_N, \infty[}{(a_N+\epsilon_N, b_N)}{t^{min}_N}\Biggl) = \exp \left(\frac{\lambda}{t^{min}_{N}} ( \primi(b_{N} ))\right).
\end{align*}

Gathering all this, we get the following:
\begin{align*}
    \p(\event^-) \geq \ &\mu([a_0, a_0+\epsilon_0]\times [\mathrm{b}_0, \mathrm{b}_0+\eta_0])\times \prod_{i=1}^{N}\lambda\epsilon_i  \mathcal{L}([b_i, b_i +\eta_i]) \\
    &\times\exp\left(-\lambda\left(\frac{-\primi(\mathrm{b}_0)}{t^{min}_1}+\sum_{i=1}^{n-1}\primi(b_i)(\frac{1}{t^{min}_{i}}-\frac{1}{t^{min}_{i+1}}\right)\right)\\
    &\times\prod_{i = 1}^{N} \exp(-\epsilon_i) \mathds{1}_{b_i\geq b_{i-1}+\epsilon_{i-1}}\mathds{1}_{t^{max}_i \leq t^{min}_{i-1}}.\\
\end{align*}
Note that, in this expression, the exponential term is continuous in all its arguments (since $\primi$ is differentiable), and the last product is a monotone function in all of the $\eta_i$ and $\epsilon_i$, that grows to $\prod \mathds{1}_{b_i\geq b_{i-1}}\mathds{1}_{t_i \leq t_{i-1}}$ when all the parameters $\epsilon_i$ and $\eta_i$ go to zero.

Moreover, our expressions are true regardless of whether the intervals $[a_i, a_i+\epsilon_i]$ and $[b_i, b_i+\eta_i]$ are open or closed. As such, by decomposing them as the unions:
\[[b_i, b_i+\eta_i] = \bigcup_{j = 0}^{n-1}\biggl[b_i+j\frac{\eta_i}{n}, b_i+(j+1)\frac{\eta_i}{n}\biggl[ \ \cup \{b_i+\eta_i\}\]
and
\[[a_i, a_i+\epsilon_i] = \bigcup_{j = 0}^{n-1}\biggl[a_i+j\frac{\epsilon_i}{n}, a_i+(j+1)\frac{\epsilon_i}{n}\biggl[ \ \cup \{a_i+\epsilon_i\},\]
and fixing a certain precision $\epsilon_i/n_0$ for the last term, we end up computing a Riemann sum in $2(n+1)$ dimensions, which yields, for any $n_0$:
\begin{align*}
\p(\event)\geq \int\lambda^N\exp\left(-\lambda\left(\frac{-\primi(h_0)}{t_1}+\sum_{i=1}^{N-1}\primi(h_i)(\frac{1}{t_{i}}-\frac{1}{t_{i+1}}\right)\right)\prod_{1\leq i\leq N}\mathds{1}_{h_i> h_{i-1}+\frac{\epsilon_{i-1}}{n_0}}&\mathds{1}_{\frac{h_i+\frac{\eta_i}{n_0}-h_{i-1}}{x_i-x_{i-1}}\geq\frac{h_i-h_{i-1}}{x_i-x_{i-1}-\frac{\epsilon_{i-1}}{n_0}} }\\\exp(-\frac{\epsilon_i}{n_0})
&\rm{d}\mu\displaystyle\bigotimes_{i=1}^{N} (\Leb \otimes \mathcal{L})(x_i, h_i).
\end{align*}

Monotone convergence for $n_0$ going to $+\infty$ finally gives us:
\[\p(\event)\geq \int g_N((x_i, h_i)\rm{d}\mu\displaystyle\bigotimes_{i=1}^{N} (\Leb \otimes \mathcal{L})(x_i, h_i).\]

However, $g_N$ is a density, and since we are on $\R^{2(N+1)}$, having the law of $(X_n, H_n)_{0\leq n\leq N}$ greater than the \textit{probability measure} $g_N\mu\displaystyle\bigotimes_{i=1}^{N} (\Leb \otimes \mathcal{L})$ on every product of intervals means that they are equal. 

As such, we have the wanted result.
\end{proof}

\begin{remark}
The Markov chain $(H_n, \mathrm{t}_n)$ has a directed evolution. Indeed, its first component is almost surely decreasing, and its second component almost surely increasing. Moreover, we have the following:
\end{remark}

\begin{proposition}\textbf{Limit of the Markov chain}\label{limite}

Let $(H_n, \mathrm{t}_n)$ be the Markov chain defined in Corollary \ref{Markov}. Then, almost surely, $H_n$ goes to $\sup \mathrm{supp}(\mathcal{L})$ and $t_n$ goes to zero when $n$ goes to infinity.
\end{proposition}
\begin{remark}
This, among other things, implies that $X_n$ almost surely goes to infinity when $n$ goes to infinity, but that could be simply deduced from the fact that the set of bases of the buildings $\{x_i\}$ is locally finite, and $(X_n)$ is an increasing sequence of points of the said set.
\end{remark}
\begin{proof}
The main idea is to notice that no tall building can be "skipped" by the blocking building sequence, which yields the first limit. Then, knowing the limit of the blocking heights, the sparsity of tall buildings yields the fact that blockage angles go to zero.\\

More formally, let us denote $S : = \sup \mathrm{supp}(\mathcal{L})$. 

On the one hand, from the definition of $\tau$, it is clear that $H_{n+1}$ is the maximum of the heights of the buildings between $X_n$ and $X_{n+1}$. As such, $H_n$ is greater than the maximum of the first $n$ from the origin (since $(X_n)$ is strictly increasing). However, this maximum goes almost surely to $S$, since it is the maximum of $n$ independent random variables of law $\mathcal{L}$. As such, we have the first convergence announced.\\

On the other hand, we have, for $A<S$ and $\alpha>0$ given,

\begin{align*}
    \p\left(t_n\geq \alpha, H_n\geq A\right) &= \int_\alpha^{+\infty}\int_{A}^{+\infty}\int_{\R_+}\int_{\R_+} \mathds{1}_{h> H, t\leq  T} \frac{h-H}{t^2}\lambda \exp\left(-\lambda \primi(H)(\frac{1}{T}-\frac{1}{t})\right)\\
    & \qquad\qquad\qquad\qquad\qquad\qquad\qquad\qquad\qquad\rm{d}\p_{h_{n-1}}(H)\rm{d}\p_{t_{n-1}}(T)\rm{d}\mathcal{L}(h)dt\\
     &\leq \int_\alpha^{+\infty}\int_{A}^{+\infty}\frac{h}{t^2}\lambda \rm{d}\mathcal{L}(h)dt\\
    &\leq \int_{A}^{+\infty}\frac{h}{\alpha^2}\lambda \rm{d}\mathcal{L}(h).
\end{align*}

As such, since $\mathcal{L}$ has a finite first moment, for $\alpha$ and $n$ fixed, $ \p\left(t_n\geq \alpha, H_n\geq A\right)$ goes to zero when $A$ goes to $S$ (indeed, monotone convergence yields $\e[H\mathds{1}_{H\geq A}]\rightarrow S\mathcal{L}(S) = 0$ if $S$ is finite or directly $0$ when $S = +\infty$).\\

Let $\alpha>0$ be fixed. For any $\epsilon>0$, we have the following :
\begin{itemize}
    \item From what we just computed, there is a certain $A< S$ such that for all $n$,  $\p\left(\mathrm{t}_n\geq \alpha\right) = \frac{\epsilon}{2}+ \p\left( H_n< A\right)$.
    \item Since the convergence of $H_n$ to $S$ when $n$ goes to infinity is almost sure, it also holds in probability. Thus, there is a certain $n_0$ starting from which $\p\left( H_n< A\right)$ is lower than $\frac{\epsilon}{2}$.
\end{itemize}
We proved here that $\mathrm{t}_n$ goes to zero in probability when $n$ goes to infinity, which, since $\mathrm{t}_n$ is almost surely non-increasing, means that $\mathrm{t}_n$ almost surely goes to zero when $n$ goes to infinity.

Hence, we have the two wanted results.
\end{proof}

\subsection{Mass-at-a-maximum-height case}\label{max_height}

So far, we stated results for the case where there is no mass on the maximum of the support of $\mathcal{L}$, since our definition of blockage does not hold in the case where such a mass exists. However, there are some ways to mend this definition that allow one to consider this maximum height case. Here are two schemes $\tau_1$ and $\tau_2$ that extend $\tau$ to a height distribution with mass on its maximum $S$ (which we refer as \textit{max-height case} in the following):
\begin{itemize}
    \item If the framework is focused on the connection to a satellite, with no restriction of range or connection to users on the ground, one can then consider that, once the maximum height is reached, no further angular improvement can be made, and as such consider the blocking building of a max height building to be itself:
    \[\tau_1(x, S) := (x, S).\]
    \item If one is interested in the interconnection between grounded users, then one might want to consider the blocking building of a max height one to be the next max height building, as an "optimal way" to relay the signal:
    \[\tau_2(x, S) := (\min\{x: (x, S)\in \landsc\},S).\] 
\end{itemize}

In the first case, we can define by convention the blockage angle for a max height building to be $0$.

\begin{remark}
One can formulate the schemes $\tau_1$ and $\tau_2$ in a deterministic context, in a way that does not depend on an external parameter being the law $\mathcal{L}$:
\begin{itemize}
    \item $\tau_1 (x, y):= \tau(x, y)$ when it is well defined and $(x, y)$ otherwise.
    \item $\tau_2 (x, y) := \min\{(a, b)\in \landsc: \frac{b-y}{a-x} = \underset{\substack{(a, b) \in \landsc \\a>x}}{\max} \frac{b-y}{a-x}\}$ when this is well defined, which happens if and only if $\underset{\substack{(a, b) \in \landsc \\a>x}}{\max} \frac{b-y}{a-x}$ is well defined.
\end{itemize}
\end{remark}

\begin{remark}
There can be other schemes in different settings. For instance in a finite range case, one can look at the \textit{furthest} building in range that has maximum height.
\end{remark}

These schemes do not change much in terms of the formulation of the evolution we gave in Theorem \ref{evo} and Lemma \ref{Markov}. 

One the one hand, $\tau_1$ changes the Markov chain so that any $(S, t)$ can only transition to a cemetery state $(S, 0)$, and the rest of the evolution stays true up to a small change: in the formulation we had as a product for the density, one has to add to each term an indicator $\mathds{1}_{H_{n-1} = S, H_{n} = S, t_n = 0}$. In other words, if one is to condition to $(X_{N-1}, H_{N-1})$, the conditional expectation of $(X_{N}, H_{N})$ has law:
\[\mathds{1}_{h>H_{N-1}}\mathds{1}_{t_{N-1}\geq t^{N}}\lambda \exp\left(-\lambda\primi(H_{N-1})(\frac{1}{t_{N-1}}-\frac{1}{t})\right)\Leb\otimes\mathcal{L} + \mathds{1}_{H_{N-1} = S}\delta_{^{N-1}, S}.\]

On the other hand, when studying the second scheme $\tau_2$, the same behaviour of cemetery state appears for the Markov chain $(H_n, t_n)$. However, the density of the evolution stays exactly the same as in Theorem \ref{evo}, where one takes the convention $\primi(S)(\frac{1}{t_i}-\frac{1}{t_{i+1}}) = \mathcal{L}(S)\times (x_{i+1}-x_i)$, or in other words, the conditional law of $(X_{N}, H_{N})$ conditioned on $(X_{N-1}, H_{N-1})$ is:
\[\left(\mathds{1}_{h>H_{N-1}}\mathds{1}_{t_{N-1}\geq t^{N}}\lambda \exp\left(-\lambda\primi(H_{N-1})(\frac{1}{t_{N-1}}-\frac{1}{t})\right)+\mathds{1}_{H_{N-1} = S}\mathds{1}_{h=S}\lambda\exp\left(-\lambda\mathcal{L}(S)(x-X_{N-1})\right)\right)\Leb\otimes\mathcal{L} .\]

In both cases, the proof can be obtained using the same method as in Theorem \ref{evo}, by simply considering the change of conditioning induced by the new schemes. Alternatively, one can also get those results from the Markov chain behaviour by noticing that, up until the hitting time of $\{S\}\times \R_+$, the existence of a maximum height does not change the transitions of the blockage chain defined in Corollary \ref{Markov}.\\

Moreover, we have the following limit behaviour for those schemes:

\begin{proposition}\textbf{Limit of the Markov chain: Max-Height case.}

Let us suppose that $\mathcal{L}$ puts non zero mass on $S$, the maximum of its support. Then, almost surely, $(H_n, t_n)$ goes to $(S, 0)$ in finite time:
\[\p\left(\exists n \in N : (H_n, t_n) = (S, 0)\right) = 1.\]
Moreover, we have the following behaviour for $(X_n)$:
\begin{itemize}
    \item Under scheme $\tau_1$, it is constant starting from the hitting time of $\{S\}\times \R_+$ (which, from the above, is almost surely finite).
    \item Under scheme $\tau_2$, it almost surely goes to $+\infty$.
\end{itemize}

\end{proposition}

\begin{proof}
The first part directly comes from the fact that until the hitting time of $\{S\}\times \R_+$, both schemes coincide, and that, in a similar fashion as what was the case for $\tau$, $\tau_2^{(n)}(x_0, h_0)$ stochastically dominates the highest of the $n$ first buildings at the right of $x_0$.

The second part simply comes from the fact that $\tau_1$ stays constant after hitting the max height, while the bases of the buildings reached through $\tau_2$ form an increasing sequence of a locally finite set.
\end{proof}

\begin{remark}
So far, excpet the convergence result, we never used the homogeneous characteristic of the PPP of the buildings' positions, nor the identically distributed characteristic of the heights. In fact, since our results rely mostly on the thinning Lemma \ref{dessous_diago}, and this lemma is still valid in the in-homogeneous case or with position-dependent marks, our methods can be adapted in a straightforward manner to those cases. Note that, in the following section, while our numerical study, which depends on a deterministic geometrical study, still holds, the qualitative vertex shift classification, which depends heavily on the unimodular aspect of the underlying graph, needs a homogeneity assumption.
\end{remark}

\section{The Eventual Relay Tree and the Backwards Problem}\label{tree}

So far, we considered the evolution of visibility and blockage experienced by a typical user. In other words, if we consider the relaying scheme $\tau$ as directed edges between the buildings (and up to adding a Poisson Process of users with null height on the line), we studied the descendants of a typical user. However, it can be of interest to understand other aspects of the graph stemming from those directed edges. For instance, a motivation and focus for this section is to understand the converse problem of the number of connections passing through a given building, which is linked to dimensioning problems for the RISs one might install on the tops of buildings. 

As such, a first part of this section is dedicated to better understanding the connection between buildings induced by the scheme $\tau$ or one of the alternative schemes $\tau_1$ and $\tau_2$. In the following, when it is not precised, we denote $\tau$ for the schemes $ \tau_2$, as it strictly extends the first defined scheme $\tau$. However, unless mentioned otherwise, the results also extend to scheme $\tau_1$, albeit with some changes in the definitions and notations.

\subsection{Eventual Relay}

In this subsection, we assume $\landsc$ to be a given landscape such that the scheme $\tau_2$ is well defined from any point.

\begin{definition}\textbf{Eventual Relay Forest.}

We say that a point $(x, y)$ is eventually relayed by a building $(a, b)$ if there is an $n\in \N$ such that
\[\tau^{(n)}(x, y) = (a, b).\]

Moreover, we define the \textit{eventual relay graph} as the directed graph $\graph = (V, E)$:
\begin{itemize}
    \item Set of vertices $V := \landsc$, the set of buildings tops;
    \item Set of (directed) edges $E := \{(b, \tau(b):b\in \landsc\}$, the edges going from every building to its blocking building.
\end{itemize}

Clearly, from the increasing characteristic of the $x$-coordinate of the blockage sequence, the only possibility for the existence of a cycle in $\graph$ is if we take the scheme $\tau_1$, in which case all cycles are of length one. By abuse of notation, we call  $\graph$ the \textit{Eventual Relay Forest}, as its cycle structure is trivial.\\
One can notice that saying that a building $b_1$ is eventually relayed by a building $b_2$ is equivalent to saying that $b_2$ is a descendant of $b_1$ in the Eventual Relay Forest.
\end{definition}

Since we are not always only interested in the building-to-building connection, we further introduce the following notation:
\begin{definition}\textbf{Zone of Eventual Relaying.}

For a building $(a, b)$, we define its \textit{zone of eventual relaying} $\zone (a, b)$ as the set of points of $\Half$ that are eventually relayed by $(a, b)$:
\[\zone (a, b) := \{(x, y)\in \Half : \exists n\in \N : \tau^{(n)}(x, y) = (a, b).\]
\end{definition}

In order to better understand the zones of eventual relaying, let us first give a characterisation of the eventual relaying by a given building. For understanding the "forward" relaying, an appropriate object to consider was the \textit{shade} of buildings from a point. Here, since the goal is to understand the backward process, a key object to consider is the \textit{Reverse Shade} of a given building:

\begin{definition}\textbf{Reverse Shade, Shadow-Cutting Building.}

Let $(\alpha, \beta)\in \landsc$ be a building. We define its \textit{reverse shade} $\rs(\alpha, \beta)$ as the shade (on the left of a) of $(\alpha, \beta)$ from $\tau(\alpha, \beta)$ ($=(\alpha_b, \beta_b)$):
\[\rs(\alpha, \beta) :=  \left\{(x, y) \in \Half: x\leq \alpha \text{ and } \frac{\beta-y}{\alpha-x} \geq \frac{\beta_b-\beta}{\alpha_b-\alpha}\right\}.\]

From there, we further define the \textit{shadow-cutting building} $\rb$ of $(\alpha, \beta)$ as the closest building to the left of $(\alpha, \beta)$ that does not lie in $\rs(\alpha, \beta)$:
\[\rb := \max\left\{(x, y) \in L: x\leq \alpha \text{ and } \frac{\beta-y}{\alpha-x} < \frac{\beta_b-\beta}{\alpha_b-\alpha}\right\}.\]
        \begin{figure}[h]\label{reverse_shade_fig}
            \centering
            \includegraphics[width=0.75\textwidth]{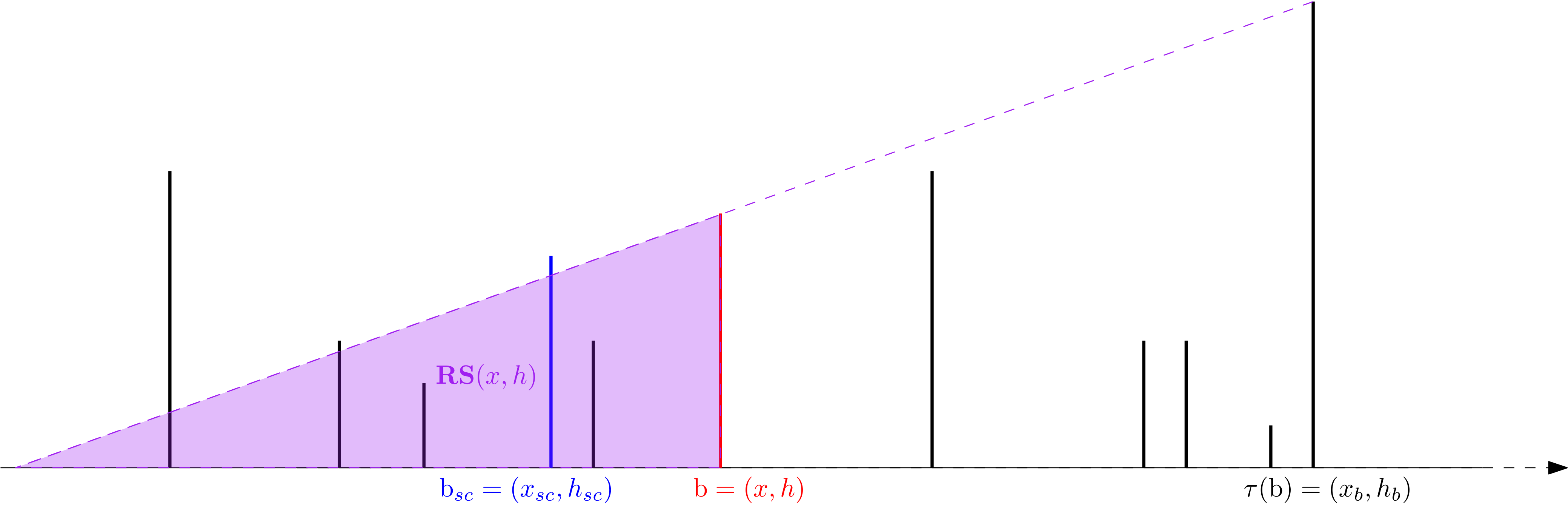}
            \caption{The shadow-cutting building can be seen as the one on which the shadow of $b$ stops.}
        \end{figure}\\
\end{definition}

\begin{remark}\label{changetau1}
If one is to consider scheme $\tau_1$, then the definitions slightly change: as we see, the main characteristic of the shadow-cutting-building is that it is "the first to not be relayed by $(\alpha, \beta)$. Hence, the proper definition for $\tau_1$ are found by interchanging the large and strict inequalities:
\begin{align*}
    \rs_{1}(\alpha, \beta) &:=   \left\{x, y) \in \Half: x\leq \alpha \text{ and } \frac{\beta-y}{\alpha-x} < \frac{\beta_b-\beta}{\alpha_b-\alpha}\right\}\\
    \rb_{1}  &:= \max\left\{(x, y) \in L: x\leq \alpha \text{ and } \frac{\beta-y}{\alpha-x} < \frac{\beta_b-\beta}{\alpha_b-\alpha}\right\}.
\end{align*}
One can notice that it does not change for Poisson building processes, outside of the max height case. All the following results (when not mentioned otherwise), save for the computational ones, extend in a straightforward way to those definitions, by changing $F$ to \[\tilde{F}: t \mapsto \mathcal{L}([0, t[).\]
\end{remark}

A first fact one can notice is that, for any building $(\alpha, \beta)$, all the buildings located between $\alpha$ and $\alpha_{sc}$ is inside the reverse shade of $(\alpha, \beta)$. This intuition of the shadow-cutting building as the one to "stop the influence" of the reverse shade is in fact embodied by the following result:

\begin{proposition}\textbf{Characterisation of the eventual relaying.}\label{charact}

For $(\alpha, \beta)$ a building and $(x, y)$ a point of $\Half$, $(x, y)$ is eventually relayed by $(\alpha, \beta)$ if and only if the two following conditions are fulfilled:
\begin{itemize}
    \item $(x, y)$ is in the reverse shade of $(\alpha, \beta)$.
    \item The base $(x, y)$, $x$, is between $\alpha$ the base of $(\alpha, \beta)$ and $a_{sc}$ the base of its shadow-cutting building: $x\in [\alpha_{sc}, \alpha]$.
\end{itemize}
In other words,
\[\zone(\alpha, \beta) = \{(x, y)\in \rs(\alpha, \beta): x> \alpha_{sc}\}.\]
       \begin{figure}[h]\label{eventual_connection_fig}
            \centering
            \includegraphics[width=0.75\textwidth]{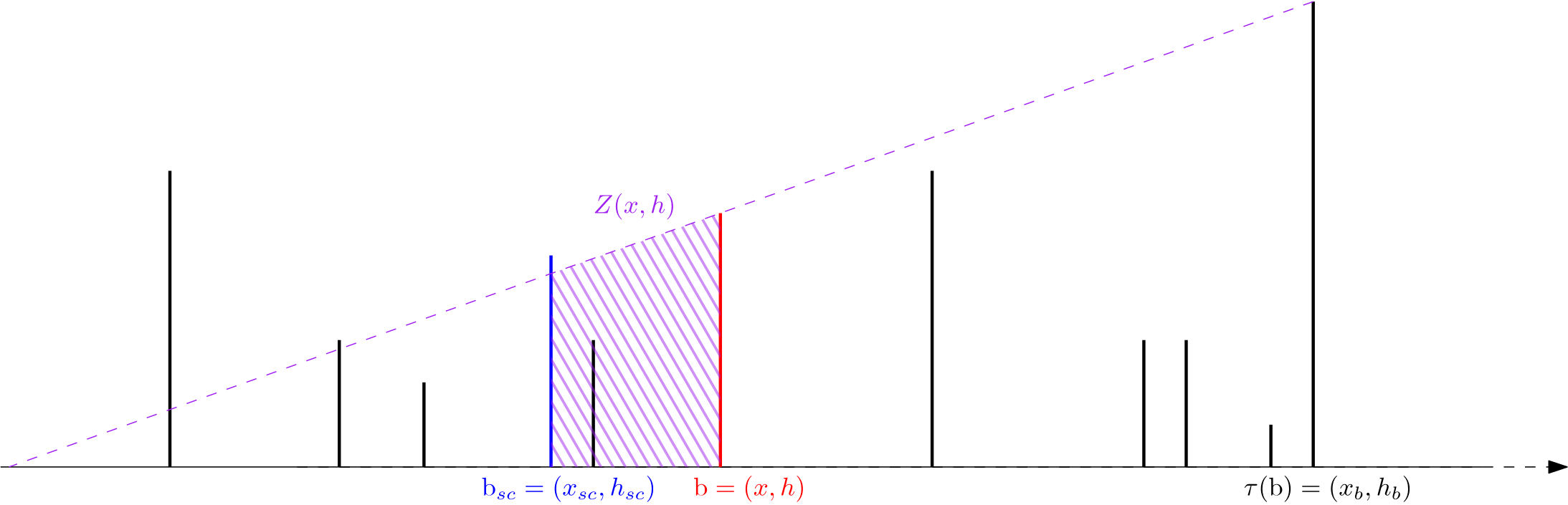}
            \caption{The buildings that are eventually relayed by $b$ are those located between $b$ and its shadow-cutting building.}
        \end{figure}
\end{proposition}

From this characterisation, we can deduce the following for schemes $\tau$ and $\tau_2$:

Let $\mathrm{b}_1 = (x_1, h_1)$ and $\mathrm{b}_2 = (x_2, h_2)$ be two buildings, with $\mathrm{b}_1$ to the left of $\mathrm{b}_2$ ($x_1\leq x_2$). Then, since the sequence of the bases of the buildings $\tau^{(n)}(b_1)$ (let us denote this sequence $(x_1^{n}$) goes to infinity, this means that there is a certain $n_0$ such that $x_1^{n_0}>x_2$. But from the characterisation of eventual relaying, this means that $x_{1, sc}^{n_0}\leq x_1\leq x_2$, which in turn means that $\mathrm{b}_2$ is eventually relayed by $\tau^{(n_0)}(\mathrm{b}_1)$. This observation leads to the

\begin{corollary}
    Under schemes $\tau$ and $\tau_2$, the Eventual Relay Forest is a tree.
\end{corollary}

As such, in the following, we use the term Eventual Relay Tree, or Relay EFT, when speaking in the context of schemes $\tau$ and $\tau_2$.\\
Let us now prove Proposition \ref{charact}:

\begin{proof}[Proof of Proposition \ref{charact}.]
Let $(\alpha, \beta)$ be a fixed building, and $t$ the tangent of its blockage angle $\theta$. The proof relies on three facts, which are formalised and proved in what follows:
\begin{itemize}
    \item A point outside of $\rs(\alpha, \beta)$ is relayed $\tau (\alpha, \beta)$ rather than by $(\alpha, \beta)$.
    \item A point left of the shadow-cutting building $\rb$ is relayed by a building outside the reverse shade before being relayed by $(\alpha, \beta)$.
    \item A point satisfying both conditions of the proposition is relayed by $(\alpha, \beta)$ before being relayed by any building on its right.
\end{itemize}
All of three being direct consequences of Proposition \ref{geomsimple}.

       \begin{figure}[h]\label{expla_eventual}
            \centering
            \includegraphics[width=0.75\textwidth]{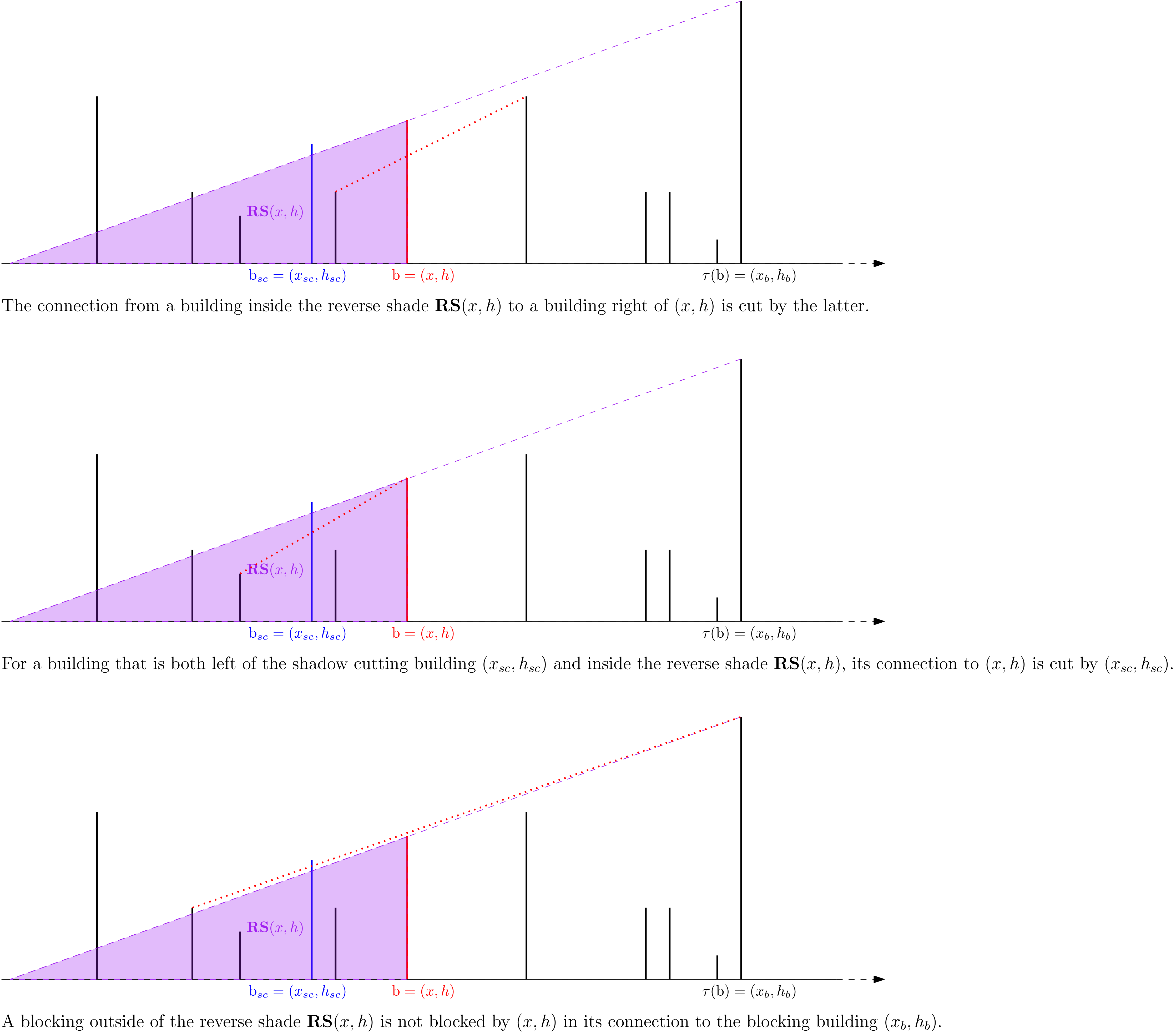}
            \caption{Representation of the three points above.}
        \end{figure}

First, let us show that, for a point outside of the reverse of $(\alpha, \beta)$, its blocking building is neither in the reverse shade of $(\alpha, \beta)$ nor $(\alpha, \beta)$ itself.

Let $(x, y)\notin \rs(\alpha, \beta)$ be given. This means that the angle $\theta '$ between $(x, y)$ and $(\alpha, \beta)$ is strictly lower than $\theta$. Thus, this means, from Proposition \ref{geomsimple}, that the angle between $(x, y)$ and $\tau(\alpha, \beta)$ is strictly between $\theta$ and $\theta '$, hence strictly greater than $\theta '$. This implies that $\tau (x, y) \neq (\alpha, \beta)$.\\
Similarly, for $(x', y')\in \rs(\alpha, \beta)$ such that $x'> x$, the angle between $(x', y')$ and $(\alpha, \beta)$ is greater than that between $(x, y)$ and $(\alpha, \beta)$, hence the fact that the angle between $(x, y)$ and $(x', y')$ is lower than the one between $(x, y)$ and $(\alpha, \beta)$. Since this is true for any building in the reverse shade of $(\alpha, \beta)$, this means that $\tau (x, y) \notin \rs(\alpha, \beta)$. This yields the fact that no building outside the reverse shade $\rs(\alpha, \beta)$ is eventually relayed by $(\alpha, \beta)$.

Now, let us show that if a point is in the reverse shade of $(\alpha, \beta)$, but left of the shadow-cutting building, then its blocking building is either on the left of the blocking building $\rb$ or the said shadow-cutting building itself.

Let $(x, y)\in \rs(\alpha, \beta)$ such that $x<\alpha_{sc}$ be given. From Proposition \ref{geomsimple}, we deduce the following:
\begin{itemize}
    \item The angle $\tilde{\theta}$ between $(x, y)$ and $(\alpha, \beta)$ is greater than $\theta$. Moreover, it is between $\theta '$ the angle formed by $(x, y)$ and $\rb$ and the angle formed by $\rb$ and $(\alpha, \beta)$. From the above observations, this means that $\tilde{\theta}$ is lower than $\theta '$.
    \item Since $\theta$ is the blockage angle of $(\alpha, \beta)$, this means that the angle between $(x, y)$ and any building to the right of $(\alpha, \beta)$ is lower than $\tilde{\theta}$, hence lower than $\theta '$.
    \item For $(x', y')\in \rs(\alpha, \beta)$ such that $x '>\alpha_{sc}$, the angle formed by $(x, y)$ and $(x', y')$ is between $\theta '$ and the ange between $\rb$ and $(x', y')$. From the observations in the previous case, this means that the angle between $(x, y)$ and $(x', y')$ is lower than $\theta '$.
\end{itemize}
Overall, this means that for any building $(a', b')$ right of $\rb$, the angle between $(x, y)$ and $(a', b')$ is lower than $\theta'$, which means that $\tau(x, y)$ is either left of $\rb$ or $\rb$ itself. This, together with the previous case, implies that no building left of the shadow-cutting building $\rb$ is eventually relayed by $(\alpha, \beta)$.

Lastly, we want to show that all points satisfying both conditions of the proposition are eventually relayed by $(\alpha, \beta)$:

Let $(x, y)\in \rs(\alpha, \beta)$ such that $x>\alpha_{sc}$ be given. We know that $\theta '$, the angle formed by $(x, y)$ and $(\alpha, \beta)$, is  greater than $\theta$. Since $\theta$ is the blockage angle of $(\alpha, \beta)$, this means that, for any building $(x ', y')$ with $x '> \alpha$, the angle formed by $(x, y)$ and $(x',  y')$ is between $\theta$ and $\theta '$, hence lower than or equal to $\theta '$. Which means that $\tau (x, y)$ cannot be right of $(\alpha, \beta)$. But since the sequence $(\tau^{(n)}(x, y))$ eventually goes to the right of $(\alpha, \beta)$, this means that there is a certain $n_0$ such that $\tau^{(n_0)}(x, y) = (\alpha, \beta)$.
\end{proof}
This characterisation of the zone of eventual relaying turns out to be a powerful tool for qualitative and quantitative study of the Eventual Relay Forests. Moreover, our results give key properties that are essential for understanding the infinite range case. A zone of eventual relaying is a connected part of the half plane, and is the right trapezoid with angle $\theta(\alpha, \beta)$ that has one of its edges be the segment between $(\alpha, 0)$ and $(\alpha, \beta)$.

However, more hypotheses need be made to derive either general graph properties on the Relay EFTs or quantitative estimates on the zones of eventual relay. As such, let us understand its implication in the case of Poisson building processes.

\subsection{Backward Problem in the Poisson case}

Thanks to Proposition \ref{charact}, we are able to get numerical estimates and to quantify some objects of interest in the study of the Relay EFT for Poisson building landscapes. However, this also allows for a more qualitative study of the said forest under more general hypothesis, which can then be further analyzed to the light of the quantitative results. Let us first give a few of these results.

In the remainder of this section, $\mathcal{L}$ is a law on $\R_+$ with finite mean, $\landsc$ is a Poisson building process of intensity $\lambda >0$ and height distribution $\mathcal{L}$.

\subsubsection*{Unimodularity, Covariant Vertex Shifts and Eternal Family Trees}

A first strong structural property of the Eventual Relay Forest in this context comes from its stationary distribution: Indeed, since it is built in a translation invariant way from a stationary marked point process, we have the following (see \cite{EFT} example $2.6$, \cite{Aldous} example $9.5$), which links our model to a well understood class of graphs, and allows us to use the powerful tool of \textit{mass transport principle}:

\begin{proposition}\label{unimodular}\textbf{Unimodularity of the Relay EFT.}

Under the Palm measure of the marked point process $\landsc$, the graph $\graph$ rooted at the "origin building" is unimodular.
\end{proposition}
In other words, if we add a building at the origin, with height distribution $\mathcal{L}$, then the Relay EFT, rooted in this new building, is unimodular. In the following, we denote by $\graph_0$ the said tree.

This property only relying on the building distribution and the stationarity of our relaying scheme means that it extends to further cases where our study of reverse shade might fail (see Section \ref{finite_range} for instance).\\
Moreover, as we saw in the first part, the  relaying schemes can be seen not only as a way to connect buildings to each other in a kind of static percolation way, but also as a dynamical system : every building has exactly one blocking building. Moreover, this way of doing is certainly covariant under directed graph isomorphism, since it is exactly following the edges of the said graph.\\
As such, the function $\tau$ can be seen as a \textit{covariant vertex shift}, which makes $\graph_0$ fall in the scope of the \textit{unimodular foil classification Theorem} (\cite{EFT}, Theorem $3.10$). 

This theorem states the following:\\
Let a \textit{foil} be a generation in the following sense: two buildings $\mathrm{b}_1$ and $\mathrm{b}_2$ are in the same foil if and only if there is an $n\in \N$ such that $\tau^{(n)}(\mathrm{b}_1) = \tau^{(n)}(\mathrm{b}_2)$. Then, for each connected component of $\graph_0$, it is in either of the three following classes:
\begin{itemize}
    \item Class $\mathcal{F}/\mathcal{F}$: The component is finite, and has finitely many foils, that are all finite. It can be decomposed as a cycle, on which all trajectories end, and finitely many finite branches, that converge to the said cycle.
    \item Class $\mathcal{I}/\mathcal{F}$: The component is infinite, but all its foils are finite. When forgetting about the directions of the edges, it is a two-ended tree, with a unique bi-infinite path, which contains exactly one vertex per foil. It can be decomposed as a bi-infinite spine along which every trajectory ends up traveling, and infinitely many finite "bushes" attached to said spine, that are converging to the spine.
    \item Class $\mathcal{I}/\mathcal{I}$: The component is infinite, and so are all of its foils. When forgetting about directions, it is a one ended tree, with no bi-infinite path. Almost surely, all vertices have finitely many trajectories passing through them, and all trajectories are vanishing towards a "direction" infinitely far away.
\end{itemize}

In our case, differentiating between the classes ends up being quite simple:
\begin{itemize}
    \item If we are not in the max-height case, then as mentioned, every trajectory is infinite, excluding the case of class $\mathcal{F}/\mathcal{F}$. Moreover, our characterisation of eventual relaying through the use of reverse shades, coupled with the fact that the blockage angle can never be $0$, means that, for any building, the number of eventual relayings by this building is a.s. finite, which in turn stops the possibility of bi-infinite path. In other words, the entire Relay EFT is of class $\mathcal{I}/\mathcal{I}$. One can then further interpret Proposition \ref{limite}, where there is a convergence towards an asymptotic, as the embodiment of the vanishing of the trajectories of this class towards the unique end of the tree.
    \item In the max height case, with scheme $\tau_1$, all trajectories are finite as we saw. This in turn means that the Eventual Relay Forest is in class $\mathcal{F}/\mathcal{F}$. The cycle structure is here trivial, as it corresponds to cycles of length $1$ as we discussed in the previous subsection. In this context, the cemetery state almost surely reached corresponds to the said cycles.
    \item In the max height case with scheme $\tau_2$, we already saw that the graph $\graph_0$ was in fact a tree. Moreover, thanks to the reverse shade characterisation, we know that on the event that the origin has maximum height, all the buildings to its left are eventually relayed by it, which in turn implies that the Eventual Relay Tree is in class $\mathcal{I}/\mathcal{F}$. The spine is exactly the set of buildings of maximum height, and the "finite bushes" are the buildings between two consecutive max-height buildings. The cemetery state convergence, coupled with the divergent sequence of blocking buildings' positions corresponds exactly to the behaviour of reaching the spine, then moving along it. 
\end{itemize}

\subsubsection*{Computation of the zone of eventual relaying}

While the results and classification we saw hold in a more general context, and not only for Poisson building processes, this i.i.d. marked Poisson characteristic allows us to derive some exact computations. Similarly to the characterisation of the zone of eventual relaying, the results in this subsection are stated for scheme $\tau_2$, but the methods can be straightforwardly adapted to scheme $\tau_1$

Two objects can be of interest here.

On the one hand, from a structural point of view on the eventual relaying graph $\graph$, one might be interested in the distribution of the buildings that are eventually relayed by a given building $\mathrm{b}_0$.

On the other hand, coming from our initial dimensioning motivation, the question of the number of connections being relayed by the top of a given building $\mathrm{b}_0$ can be investigated.

For the first problem, from the fact that the zone of eventual relaying is a trapezoid, as well as the strong Markov property for Poisson Point Processes, the buildings that are eventually relayed by $\mathrm{b}_0$ form an inhomogenous Poisson Point Process which has the law of the Poisson building process seen as a PPP on $\Half$, restricted to the zone of eventual relaying of $\mathrm{b}_0$ (see Corollary \ref{desc_build}):

As such, the distribution of the descendants of $\mathrm{b}_0$ is completely determined by $\theta$, the blockage angle of $\mathrm{b}_0$, the height of $\mathrm{b}_0$ and the horizontal distance between $\mathrm{b}_0$ and its shadow-cutting building.

For the second problem, if we are to consider a point process of users that is independent of the building process, understanding the distribution of the users that are eventually relayed by $\mathrm{b}_0$ is the same as understanding the vertical projection of $\zone(\mathrm{b}_0)$. From the characterisation of the said zone, this is exactly equivalent to understanding the horizontal distance between $\mathrm{b}_0$ and its shadow-cutting building.

For both problems, a key quantity is the ground length of $\zone(\mathrm{b}_0)$ conditioned on the height and blockage angle of $\mathrm{b}_0$. Since the blocking angle is independent from the buildings to the left of $\mathrm{b}_0$ thanks to the Markov property, the next result is of great interest for the study of the Poisson case:

\begin{theorem}\textbf{Length of the relaying zone.}\label{length}

Let $\landsc '$ be a given landscape, and $(\alpha, \beta)\in \landsc'$ be a building. Denote by $\theta$ the blockage angle of $(\alpha, \beta)$ in $\landsc'$, and let $\tilde{\landsc}$ be a random landscape given by:
\begin{itemize}
    \item On the right of $(\alpha, \beta)$, it is equal to $\landsc'$: $\tilde{\landsc}\cap \{(x, y): x\geq \alpha\} = \landsc' \cap \{(x, y): x\geq \alpha\}$.
    \item On the left of $(\alpha, \beta)$, it is a Poisson building process of intensity $\lambda$ and height distribution $\mathcal{L}$: $\tilde{\landsc}\cap \{(x, y): x> \alpha\} \overset{(d)}{=}\landsc \cap \{(x, y): x> \alpha\}$.
\end{itemize}
Then, in $\tilde{\landsc}$, the ground length of $\zone(\alpha, \beta)$, that is $\ell_\theta(\zone(\alpha, \beta)) := \min(\alpha-\alpha_{sc}, \frac{\beta}{\tan\theta})$, has its CDF equal to:
\[t\mapsto 1- \mathds{1}_{t\leq \frac{\beta}{\tan\theta}}\exp\left(-\frac{\lambda}{\tan\theta}(\primi(\beta)-\primi(\beta-t\tan\theta))\right).\]
\end{theorem}

\begin{remark}
In the case where $\theta = 0$, this stays true, as long as we take the convention that $-\frac{\lambda}{\tan\theta}(\primi(\beta)-\primi(\beta-t\tan\theta)) = -\lambda t (1-F(\beta))$, and consider the ground length as a $[-\infty, +\infty]-$valued random variable.
\end{remark}

\begin{proof}
Instead of looking at the ground length directly, we are interested in the position of the leftmost point in $\zone(\alpha, \beta)$, $\max(\alpha_{sc},\alpha-\frac{\beta}{\tan\theta})$, which is exactly $\alpha- \min(\alpha-\alpha_{sc}, \frac{\beta}{\tan\theta})$.\\
On the one hand, it is clear that the said maximum is almost surely in $[\alpha-\frac{\beta}{\tan\theta}, \alpha]$. Moreover, for $t\in [\alpha-\frac{\beta}{\tan\theta}, \alpha]$, having $\alpha_{sc} \leq t$ is equivalent to having every building building with $x$-coordinate in $[t, \alpha[$ below the diagonal $\{(x, y)\in \Half, y = (x-(\alpha-\frac{\beta}{\tan \theta})\tan \theta$. Hence, Lemma \ref{dessous_diago} yields that
\[\p(\alpha_{sc} \leq t) = \exp\left(-\frac{\lambda}{\tan \theta}(\primi(\beta)-\primi((t-(\alpha-\frac{\beta}{\tan \theta}))\tan\theta))\right).\]
This means that for $t\in \R_+$, 
\[\p(\min(\alpha-\alpha_{sc}, \frac{\beta}{\tan\theta})<t) = 1- \mathds{1}_{t\leq \frac{\beta}{\tan\theta}}\exp\left(-\frac{\lambda}{\tan\theta}(\primi(\beta)-\primi(\beta-t\tan\theta))\right).\]
The semi continuity of the CDF then allows us to conclude that this is also the expression of the CDF.
\end{proof} 

Moreover, having the distribution of $\max(\alpha_{sc},\alpha-\frac{\beta}{\tan\theta})$ allows us to also get a control on the distribution of the buildings that are eventually relayed by the building $(\alpha, \beta)$ is the context of the proposition.

Keeping the notations from the proposition, we have:

\begin{corollary}\textbf{Buildings being relayed by a given building.}\label{desc_build}
    
Conditional on $\zone(\alpha, \beta)$ the buildings in $\zone(\alpha, \beta)$ form an inhomogeneous PPP of intensity
\[\tilde{\lambda} = \lambda \rm{d}x\otimes \mathcal{L}. \]
The average number of such buildings is equal to:
\[\lambda\e[\ell_\theta(\zone(\alpha, \beta))] - \left(1-\exp(-\frac{\lambda}{\tan\theta}(\primi(\beta)-\primi(0)))\right).\]
\end{corollary}
\begin{proof}
The first part is the strong the Markov property applied to the inhomegeneous Poisson Point Process of the top of the buildings in $\Half$. The latter has intensity $\tilde{\lambda} = \lambda \rm{d}x\otimes \mathcal{L}$, so using the stopping set of the points above $\zone(\alpha, \beta)$: $\{(x, y)\in \Half:x\geq \alpha_{sc} \text{ and }y\geq \beta-(\alpha-y)\tan\theta\}$ (while this is not a compact per se, the horizontal point process is, so up to applying an increasing continuous bijection from $\R+$ to a bounded set, we can view the marks as having values in a compact, enabling us to use the Markov property on this set). This, together with the strong Markov property on the top line $\{(x, y)\in \Half:x\geq \alpha_{sc} \text{ and }y= \beta-(\alpha-y)\tan\theta\}$ yields the first wanted result.

For the second one, it comes down to computing the average number of points of an inhomogeneous PPP: Let $N(\zone(\alpha, \beta))$ be the number of buildings in $\zone(\alpha, \beta)$, and $\nu$ be the distribution of $\max(\alpha_{sc},\alpha-\frac{\beta}{\tan\theta})$. We have:
\begin{align*}
    \e[N(\zone(\alpha, \beta))] &= \e[\e[N(\zone(\alpha, \beta))| \alpha_{sc}]]\\
    &= \int_{t = \alpha-\frac{\beta}{\tan\theta}}^{\alpha}\int_{x=t}^{\alpha} \lambda F(\beta-(\alpha-x)\tan\theta)\rm{d}x\rm{d}\nu(t)\\
    &= \int_{x = \alpha-\frac{\beta}{\tan\theta}}^{\alpha}\int_{t=\alpha-\frac{\beta}{\tan\theta}}^{x} \rm{d}\nu(t)\lambda F(\beta-(\alpha-x)\tan\theta)\rm{d}x\\
    &= \int_{\alpha-\frac{\beta}{\tan\theta}}^{\alpha}\lambda\left(1-(1- F(\beta-(\alpha-x))\tan\theta)\right)\exp\left(-\frac{\lambda}{\tan\theta}(\primi(\beta)-\primi(\beta-(\alpha-x)\tan\theta))\right)\rm{d}x\\
    &= \lambda\e[\ell_\theta(\zone(\alpha, \beta))] - \int_{0}^{\beta}\frac{\lambda}{\tan\theta}\left(1- F(u)\right)\exp\left(-\frac{\lambda}{\tan\theta}(\primi(\beta)-\primi(u)\right)\rm{d}u\\
    &= \lambda\e[\ell_\theta(\zone(\alpha, \beta))] - \left(1-\exp(-\frac{\lambda}{\tan\theta}(\primi(\beta)-\primi(0)))\right),
\end{align*}
which is the announced result.
\end{proof}

Furthermore, in the case of Poisson building processes, the landscape to the right of a given point is independent from the landscape to the left of the said point, thanks to the spatial Markov property. Hence, we can compute the following:

\begin{corollary}\label{typical} \textbf{Users being relayed by a given building.}

Let $(\alpha, \beta)\in \Half$ be a given point such that $\mathcal{L}([\beta,  +\infty[)>0$, and consider the landscape $\tilde{\landsc} := \landsc\cup \{(\alpha, \beta)\}$. 

Then, the expected length of the zone of eventual relaying of $(\alpha, \beta)$ is equal to 
$$\e[\ell(\zone(\alpha, \beta))] =\left\{
    \begin{array}{ll}
        \displaystyle\int_{0}^{\beta} -\frac{1}{\lambda}\frac{\primi(\beta)}{\primi(u)^2}du & \mbox{if } \beta<S \\
        +\infty & \mbox{if }\beta=S,
    \end{array}
\right.$$
where $S$ denotes the supremum of the support of $\mathcal{L}$. 
In particular, under the Palm measure, the average length of the zone of eventual relaying for a typical building is infinite.
\end{corollary}
\begin{remark}
This result is for scheme $\tau_2$ and, of course, does not hold for scheme $\tau_1$ in the max height case. However, mending our methods in the ways we gave in Remark \ref{changetau1}, we get similar results, but this time with finite expected length under the Palm measure:
$$\e[\ell(\zone(\alpha, \beta))] =\left\{
    \begin{array}{ll}
        \displaystyle\int_{0}^{\beta} -\frac{1}{\lambda}\frac{\primi(\beta)}{\primi(u)^2}\rm{d}u & \mbox{if } \beta<S \\
       \frac{1}{\lambda \mathcal{L}(S)} & \mbox{if }\beta=S.
    \end{array}
\right.$$
\end{remark}

\begin{proof}
Since $\beta=S$ implies that the blockage angle is $0$, this case is directly implied by Proposition \ref{length}. Let us treat the case of $\beta<S$.\\
Using the spatial Markov property for the Poisson building process, the buildings to the right of $A$ and those to the left of it form independent Poisson building processes of intensity $\lambda$ and height distribution $\mathcal{L}$. Then, using the fact that $\e[\ell(\zone(\alpha, \beta))] =\e[e[\ell_\theta(\zone(\alpha, \beta))|\theta]]$, we get, using Proposition \ref{length}:
\begin{align*}
    \e \left[ \ell(\zone(\alpha, \beta))\right] &= \int_{t\geq 0}\int_{h> \beta}\int_{l\geq 0}\lambda\frac{h-\beta}{t^2}\exp \left(\frac{\lambda}{t}\primi(\beta)\right)\mathds{1}_{l\leq \frac{\beta}{t}}\exp \left(-\frac{\lambda}{t}(\primi(\beta)-\primi(\beta-lt))\right) dl \rm{d}\mathcal{L}(h) dt.
\end{align*}
Making the change of variables $u = \beta-lt$, this can be rewritten as:
\[\int_{t\geq 0} \int_{h> \beta}\int_{u=0}^{\beta} \frac{-\lambda (h-\beta)}{t^3}\exp\left(\frac{\lambda}{t}\primi(u)\right)  du \rm{d}\mathcal{L}(h) dt.\]
Using Fubini's theorem, this yields:
\begin{align*}
    \e \left[  \ell(\zone(\alpha, \beta))\right]&= \int_{u=0}^{\beta} \int_{h>\beta}\frac{h-\beta}{\lambda}\frac{1}{\primi(u)^2}\rm{d}\mathcal{L}(h)\rm{d}u\\
    &= \int_{u=0}^{\beta} \int_{h>\beta}\int_{w=\beta}^{h}\frac{1}{\lambda}\frac{1}{\primi(u)^2}dw \rm{d}\mathcal{L}(h)\rm{d}u\\
    &= \int_{0}^{\beta} -\frac{1}{\lambda}\frac{\primi(H)}{\primi(u)^2},
\end{align*}
which is the wanted result. For the infinite average under the Palm measure, one could use the above last result to compute it, integrating it against the distribution $\mathcal{L}$. However, it might be more direct to recall the unimodular characteristic of the Eventual Relay Tree when looking at a typical building under the Palm measure, allowing us to apply the mass transport principle to the function
\[f:(\mathcal{G}, o, v)\mapsto \mathds{1}_{o \text{ is eventually relayed by } v}.\]
This indicates that, under the Palm measure, the average number of buildings by which the origin is eventually relayed, which is infinite in scheme $\tau_2$, is equal to the average number of buildings that are eventually relayed by it. However, from Corollary \ref{typical}, this number is lower than $\lambda$ times the average length of the typical user's zone of eventual relaying, which completes our proof.
\end{proof}
\section{Finite range case: extensions and limits}\label{finite_range}

So far, a main component of our study has been the shapes of the shades and its avatars, that is the increasing characteristic of the blocking building's height, the decreasing characteristic of the blockage angle, the Markovian property, the maximality of the blocking building's shade or the minimality of the reverse shade. However, this property relies on the fact that the blocking angle is a global maximum, or in other words that a user is allowed to be relayed arbitrarily far: our schemes have infinite range. \\
In this section, we investigate a possible extension of our results to a case where the range of the signals is limited. We first study a specific type of range in which explicit formulas for generalisation are easier to state, and then give the methods and qualitative results for general definitions of range.

\subsection{Finite Horizontal Range}

Here, we consider a way of characterising the range that is simpler than the Euclidean distance. In the following, we consider the horizontal "distance" between two points of $\Half$ to be the distance between their projections on the $x-$axis. Looking at this kind of distance can be seen as an approximation in the case where the height variation between buildings is negligible in front of the inter-building distance. In the following, we let $R>0$ be the \textit{range} of the relay.

\subsubsection*{Change of the geometry}

Let us first reformulate the problem to fit this kind of finite range. Unless mentioned otherwise, we keep the formalism defined in the first section. 

Let us define the relaying scheme:

\begin{definition}\textbf{Blockage in the finite range case.}

For $\landsc$ a landscape, and $(x, y)\in \Half$ a point, the blocking building of $(x, y)$ in $\landsc$ is, when it is well defined, 
\[\tau (x, y, \landsc) = (x_b, y_b) := \underset{\substack{(a, b) \in \landsc \\a\in ]x, x+R]}}{\text{argmax}} \frac{b-y}{a-x}.\]
And $\tau(x, y, \landsc) = (-\infty, -\infty)$ when there is no building in $]x, x+R]$.

We define the blockage angle $\theta$ and its tangent $\mathrm{t}$ similarly to the infinite range case (with the convention that $\mathrm{t} = -\infty$).
\end{definition}

\begin{remark}
This scheme is not the only possible one. This corresponds to being relayed by the \textit{furthest} visible building, but one could for instance be relayed by the highest visible one.
\end{remark}

Under this scheme, the shades, while defined in a different manner, are still objects of interest:

\begin{definition}\label{horiz_fin_shade}\textbf{Finite Range Shades.}

Let $\landsc$ be a landscape, $(x_1, y_1)$ a point of $\Half$, and $(a, b)\in \landsc$ a building. We call the shade of $(a, b)$ from $(x_1, y_1)$ the set of points right to $a$ that form a lower slope with $(x_1, y_1)$ than $(a, b)$ does, or that are too far away from $(x_1, y_1)$:
\[S^{ (x_1, y_1)}_{(a, b)} = \left\{ (x, y)\in \Half:x_1 \in [a, x_1+R],\  \frac{y-y_1}{x-x_1}\leq \frac{b-y_1}{a-x_1}\right\}\cup \{(x, y)\in \Half:x > x_1+R\}.\]

\begin{figure}[h]\label{horiz_fin_shade_fig}
            \centering
            \includegraphics[width=0.75\textwidth]{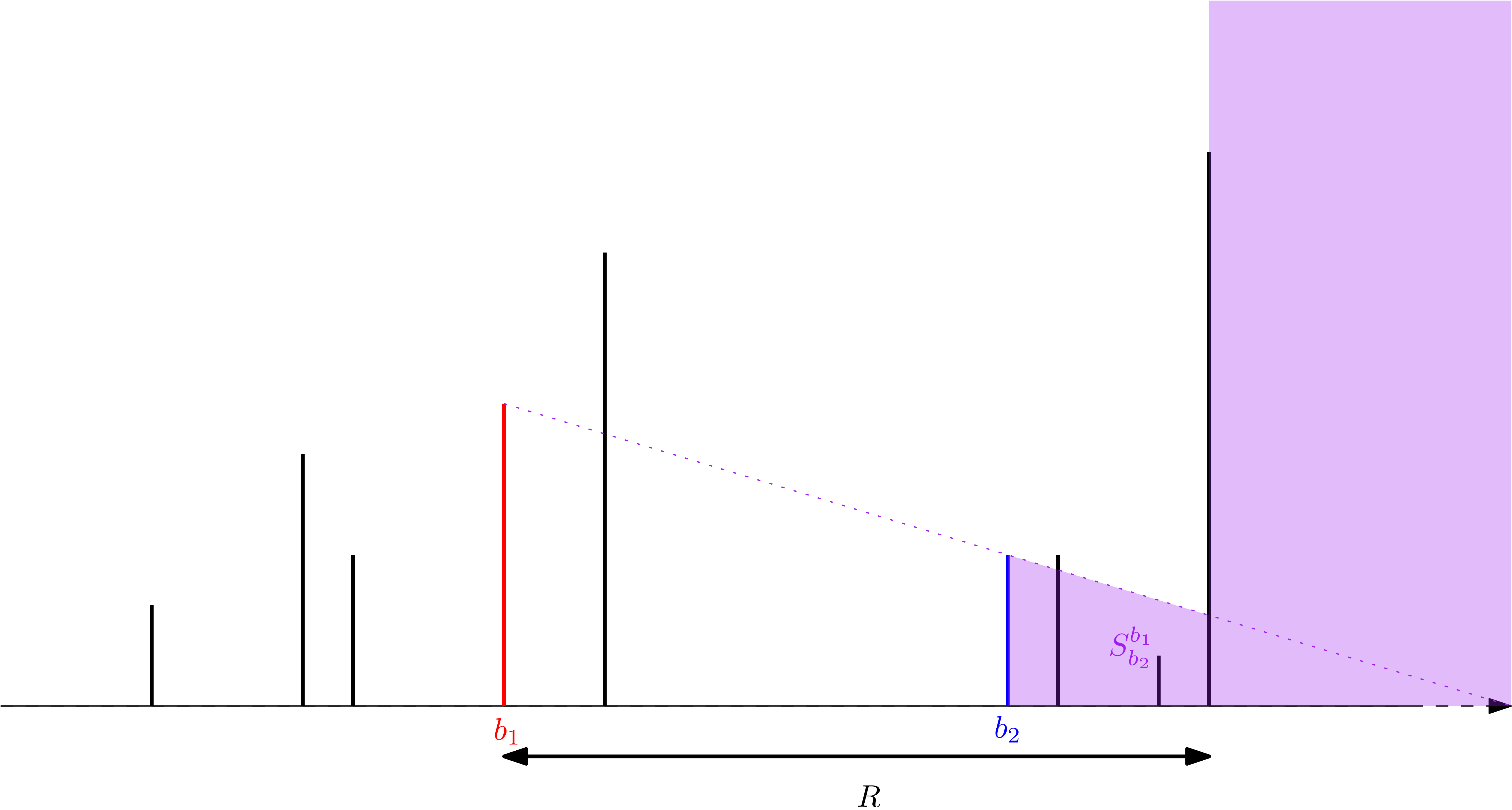}
            \caption{One of the main differences with the infinite range setting, that will be the source of the more complex results, can be seen as the fact that the shades in this context are non convex}
\end{figure}
\end{definition}

\begin{remark}
One can notice that for this specific definition of range, we can retrieve Corollary \ref{increase}: the shades of the blocking buildings are still decreasing for the inclusion.
However, this time they are not only characterised by the height of the building and the angle to it, so the Markovian characteristic is dependent on a third parameter, being the distance between buildings. This property is still true for the scheme which chooses the highest visible building.
\end{remark}

The definition that changes the most in this setting is that of the reverse shade. Indeed, the main characteristic $\rs (a, b)$ is that it is the complement of the set of points that would skip $(a, b)$ in the relaying scheme. Taking into account the range, which translates into a non convexity of the shades, we get the following:

\begin{definition}\textbf{Reverse Shade the in finite range case.}

Let $\landsc$ be a landscape and $(a, b)\in \landsc$ a building. We define the reverse shade of $(a, b)$ as \textit{the intersection of all the shades of $(a, b)$ from buildings visible from $(a, b)$, cut off at distance $R$ from $a$}:
\begin{align*}
\rs (a, b) := \displaystyle\bigcap_{\substack{(a_1, b_1) \in \landsc \\ a_1\in ]a, a+R]}} \Biggl(\{(x, y) \in \Half: (x\geq a_1-R &\text{ and } \frac{b-y}{a-x} \geq \frac{b_1-b}{a_1-a}\}\\&  \cup\{(x, y)\in \Half:a-Rx <a_1-R\}\Biggl).
\end{align*}

\begin{figure}[h]\label{Reverse_shade_fini_horiz}
            \centering
            \includegraphics[width=0.75\textwidth]{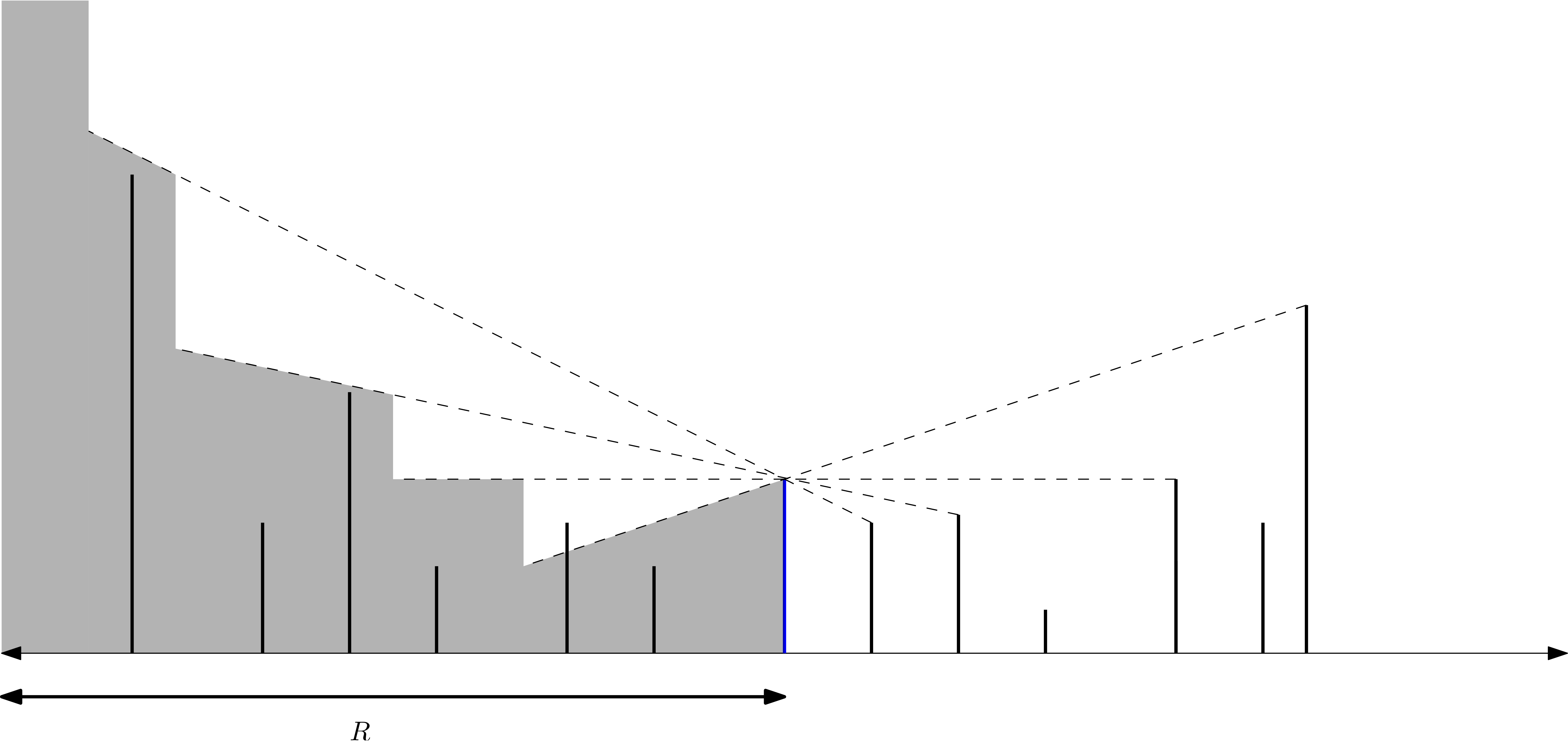}
            \caption{The reverse shade is characterized by a succession of upwards jumps and then increasing slopes.}
\end{figure}
\end{definition}

\begin{remark}\label{RS_fin}
This difference of the reverse shades also changes the characterisation of eventual relays: our result mainly hinged on the fact any building on the left of the blocking building is eventually relayed by a building outside of the reverse shade (or in other words on a kind of inclusion of reverse shades). However, here, the characterisation gets more complicated. Indeed, it can be formulated in terms of a sort of inclusion-exclusion form: it is the reverse shade, to which we successively either add or retract the reverse shades of the buildings to the left of it, depending on whether they stand or not in the current set. This characterisation holds, but is heavy computationally, and does not have simplifying properties as in the infinite-range case (for instance, it is in general not connected).

\begin{figure}[h]\label{connection_area_fini_horiz}
            \centering
            \includegraphics[width=0.3\textwidth]{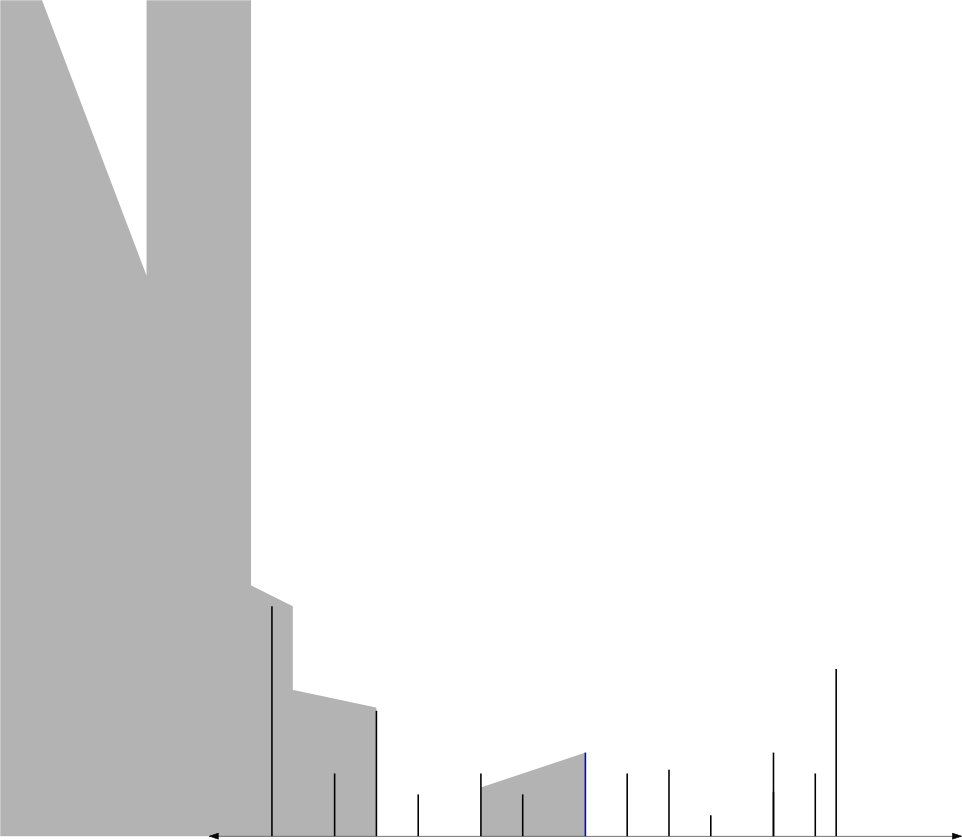}
            \caption{In the previous figure, considering only the present buildings, the relaying zone would look as above.}
\end{figure}
\end{remark}

\subsubsection*{Results in a random landscape}

In the following, we consider $\landsc$ a Poisson building process with intensity $\lambda>0$ and height distribution $\mathcal{L}$. For the sake of clarity, we assume that $\mathcal{L}$ does not have mass on the supremum of its support. The possible generalisations to this case are a straightforward adaptation of subsection \ref{max_height}.

As far as the evolution of the blockage is concerned, out methods from section \ref{multihop} still work the exact same way. This allows us to get the following:
\newpage
\begin{proposition}\label{evo_finite}\textbf{Evolution of blockage: horizontal finite range.}

Recall the notations of Theorem \ref{evo}. Let $N$ be a fixed integer. The restriction of the law of $(X_n, H_n)_{0\leq n\leq N}$ to $\Half^{n+1}$, or in other words, the restriction of the law of the evolution of blockage to the event of always finding a blocking building in range has density with respect to the product measure $\mu\displaystyle\bigotimes_{i=1}^{N} (\Leb \otimes \mathcal{L})$, which is given by
\begin{align*}
g_{N, R}: (x_n, h_n)_{0\leq n\leq N} \mapsto \lambda^N\exp\Biggl(-\lambda\Biggl(&\frac{\primi(h_{n-1}+Rt_n)}{t_n}-\frac{\primi(h_0)}{t_1}+\sum_{i=1}^{n-1}\primi(h_i)(\frac{1}{t_{i}}-\frac{1}{t_{i+1}})\\
&+ \sum_{i=1}^{n-1} (\frac{\primi(h_i + r_i (t_i \wedge t_{i+1}))}{t_i\wedge t_{i+1}}-\frac{\primi(h_i +r_i t_{i+1})}{t_{i+1}})\Biggl)\Biggl)\\
\prod_{1\leq i\leq n}&\mathds{1}_{x_i\in ]x_{i-1}, x_{i-1}+R]}\mathds{1}_{t_i\leq t_{i-1} \text{ or } x_i > x_{i-2}+R },
\end{align*}
where:
\begin{itemize}
    \item The $t_i$ are defined to be the slopes from $(x_{i-1}, h_{i-1})$ to $(x_i, h_i)$ (with the convention $t_0 = +\infty$).
    \item The $r_i$ are defined as the \textit{leftover signal range}: $r_i := R-(x_i-x_{i+1})$.
\end{itemize}
\end{proposition}

\begin{remark}
The interpretation of this density can be done in the same way as for the infinite range case. The main difference comes from the term
\[\exp\left(-\lambda\sum_{i=1}^{n-1} (\frac{\primi(h_i + r_i (t_i \wedge t_{i+1})}{t_i\wedge t_{i+1}}-\frac{\primi(h_i +r_i t_{i+1})}{t_{i+1}}\right)\]
which accounts for the fact that the variables $t_i$ are not always decreasing: sometimes some information from a blockage carries on through the later ones. However, for the same reasons as the increase of information, this can only happen if the distance between $X_i$ and $X_i+2$ is greater than $R$.

The addition of the term
\[\exp\left(-\lambda\frac{\primi(h_{n-1}+Rt_n)}{t_n}\right)\]
comes from the fact that the last blockage angle is only true up to distance $R$ to the second to last blocking building.

Moreover, the indicator $\mathds{1}_{t_i\geq t_{i-1} \text{ or } x_i > x_{i-2}+R }$ is just the indicator that $(x_i, h_i)$ is in the shade of $(x_{i-1}, h_{i-1})$ from $(x_{i-2}, h_{i-2})$.

\end{remark}

Again, one can find a Markov behaviour from this formulation:
\begin{corollary}\label{Markov_f}\textbf{Markov behaviour in horizontal finite range.}

For $n\in\N$, let $\tilde{x}_n:= X_n-X_{n-1}$ denote the length of the $n-$th hop.

The sequence $(\tilde{x}_n, \mathrm{t}_n, H_n)$ is a Markov chain with kernel $\mathrm{VIS}_R$ given by:
    \begin{align*}
        &\forall (\tilde{X}, T, H)\in \Half, \forall A \in \mathcal{B}(\R\times\Half), \mathrm{VIS}_R((\tilde{X}, T, H), A) = \int_A vis^{\tilde{X}, T, H}(t, h) \delta_{\frac{h-H}{t} }(\tilde{x})\rm{d}\Leb(t) \rm{d}\mathcal{L}(h)
        \end{align*}
    \begin{align*}
        \forall \tilde{X}, T, H, t, h \in \R_+, vis^{\tilde{X}, T, H} (t, h) :=\frac{h-H}{t^2}\lambda \exp\Biggl(-\lambda \Biggl(&\frac{\primi(H)-\primi(H+(R-\tilde{X})T)}{T}-\frac{\primi(H+Rt)-\primi(H)}{t}\\
        &+\frac{\primi(H + (R-\tilde{X} (t \wedge T)}{t\wedge T}-\frac{\primi(H +(R-\tilde{X} t)}{t}\Biggl)\Biggl)\\
        &\mathds{1}_{\frac{h-H}{t}\in ]0, R], T\geq t \text{ or } \tilde{X}+ \frac{h-H}{t}> R } ,
    \end{align*}
    and 
    \begin{align*}
        &\mathrm{VIS}_R((\tilde{X}, T, H), (-\infty, -\infty, -\infty)) = \exp\Biggl(-\lambda \Biggl(R - \frac{\primi(H)-\primi(H+(R-\tilde{X})T)}{T}\Biggl)\Biggl)\\
        &\mathrm{VIS}_R((-\infty, -\infty, -\infty), (-\infty, -\infty, -\infty)) = 1.
    \end{align*}
\end{corollary}

\begin{proof}
We give here two possible outlines. 

On the one hand, one can start by proving Theorem \ref{evo_finite} the same way as Theorem \ref{evo}, by simply bounding the desired probability on products of intervals. The only difference is the last step, where we justified that our candidate for the density $g_n$ was indeed the wanted density since it was lower than it and had total mass one. Here, since we are not interested in a proper density, but a restriction, we cannot use such a method relying on total mass. A simple possibility is to find a lower bound for $g_{N, R}$, which we can do in a very similar fashion as in the original proof: simply consider the greatest possible angles instead of the lowest, and this is a straightforward adaptation that converges to the same limit.

On the other hand, one can also start by proving Corollary \ref{Markov_f} using the spatial Markov property for the Marked PPP on the stopping set
\[\displaystyle\bigcup_{0\geq i\geq N}\{(x, y)\in [X_i, X_i+R]\times [0, S]: y\geq X_i + \mathrm{t}_i(x-X_i)\}.\]
In the case where the supremum of the possible heights $S$ is infinite, the last set not compact, but can be made one by a mapping the possible heights to a compact in a monotonic way.
\end{proof}

\subsection{Study of the zone of eventual connection}\label{ev_co_fini}
As we saw through Remark \ref{RS_fin}, the characterization for eventual connection in finite range makes the eventual relay zone less available in this context. However, thanks to the use of mass transport principle, we can still recover information about the 

For a given landscape, one can define the relay forest in finite range the same way we did previously, by simply considering an edge between each building and its blocking building. As in the infinite range case, the stationarity of Poisson Building process with respect to the horizontal translations makes said forest unimodular when rooted at the origin under Palm measure (see Proposition \ref{unimodular}). This allows us to use mass transport principle to deduce the following result:

\begin{proposition}\label{mass_transport_save}\textbf{Length of the typical relaying zone.}

Let us consider $H$ a random variable distributed following $\mathcal{L}$. Let us further consider the landscape $\tilde{\landsc} := \landsc\cup \{(0, H)\}$. Then expectation of the length of the zone of the typical eventual relaying $\zone (0, H)$ is 
\[\e[\ell(\zone(0, H))] = \frac{1}{\lambda}\e[T],\]
where $T$ is the hitting time of the cemetery state $(-\infty, -\infty, -\infty)$ for the Markov chain $(\tilde{x}_n, \mathrm{t}_n, H_n)$ started at 
$(0, 0, +\infty)$:
\[\e[T] = \sum_{n\geq 1} \displaystyle\int_{\Half^{n+1}} g_n \delta_{(0, 0)}\displaystyle\bigotimes_{i=1}^{n} (\Leb \otimes \mathcal{L}).\]
\end{proposition}

\begin{proof}
This comes from a simple mass transport principle argument: Let us consider $\Phi = \{x_i\}$ a Poisson point process of intensity $\lambda'>0$ on $\R$ independent from $\tilde{\landsc}$. Then, by definition,
\begin{align*}
\e[\ell(\zone(0, H))] &= \frac{1}{\lambda'}\e[\Phi\cap \zone(0, H)]\\
&= \frac{1}{\lambda'}\e[\# \{x\in \Phi: (x, 0)\text{ is eventually relayed by } (0, H)\}].
\end{align*}
However, the superposition of PPPs allows us to evaluate this quantity.

Consider the landscapes $\landsc_1$, $\tilde{\landsc}_1$ and $\tilde{\landsc}_2$ defined as follows:
\begin{align*}
    &\tilde{\landsc}_1 := \tilde{\landsc}\cup \{(x, 0): x\in \Phi\}\\
    &\tilde{\landsc}_2 := \landsc\cup \{(x, 0): x\in \Phi\}\cup \{(0, 0)\}.
\end{align*}
Clearly, $\landsc_1$ is a Poisson building process of intensity $\lambda +\lambda'$ and with height distribution
\[\frac{\lambda}{\lambda+\lambda'}\mathcal{L}+\frac{\lambda'}{\lambda + \lambda'} \delta_0.\]
Where by abuse of notation, we consider buildings with height $0$. By convention, we consider no relaying by such buildings: we consider the extension of the scheme to be
\[\tau (x, y, \landsc) = (x_b, y_b) := \underset{\substack{(a, b) \in \landsc \\a\in ]x, x+R]\\b>0}}{\text{argmax}} \frac{b-y}{a-x}.\]
As such, if $\omega$ denotes a Bernoulli random variable of parameter $\frac{\lambda}{\lambda+\lambda'}$ independent from everything we defined so far, the landscape
$$\tilde{\landsc'}:=\left\{
    \begin{array}{ll}
        \tilde{\landsc}_1 & \mbox{if } \omega = 1 \\
        \tilde{\landsc}_2 & \mbox{if }\omega=0
    \end{array}
\right.$$
follows the Palm measure of the Poisson Building process $\landsc'$. As such, the relay forest is unimodular when rooted in $(0, \omega H)$. Moreover, for any real function $f$ covariant with rooted graph isomorphisms
\begin{align*}
    \e[f(\mathcal{G}( \tilde{\landsc}'))]  &= \frac{\lambda}{\lambda+\lambda'}\e[f(\mathcal{G}( \tilde{\landsc}_1))] + \frac{\lambda'}{\lambda+\lambda'}\e[f(\mathcal{G}( \tilde{\landsc}_2))].
\end{align*}
Using the mass transport principle, we get on the one hand:
\begin{align*}
    \e[\ell(\zone(0,\omega H, \tilde{\landsc}'))] &= \e\left[\sum_{(x, h)\in \tilde{\landsc}'} \mathds{1}_{h = 0, (x, h) \text{ is eventually relayed by } (0, \omega H)}\right]\\
    &= \frac{\lambda}{\lambda+\lambda'}\e\left[\sum_{(x, h)\in \tilde{\landsc}_1} \mathds{1}_{h = 0, (x, h) \text{ is eventually relayed by } (0,  H)}\right]+ 0
\end{align*}
and on the other hand:
\begin{align*}
    \e\left[\sum_{(x, h)\in \tilde{\landsc}'} \mathds{1}_{h = 0, (x, h) \text{ is eventually relayed by } (0, \omega H)}\right]&= \e\left[\sum_{(x, h)\in \tilde{\landsc}'} \mathds{1}_{\omega H= 0, (0, 0) \text{ is eventually relayed by } (x, h)}\right]\\
    &= 0+ \frac{\lambda'}{\lambda+\lambda'}\e\left[\sum_{(x, h)\in \tilde{\landsc}_2} \mathds{1}_{(0, 0) \text{ is eventually relayed by } (x,  h)}\right].
\end{align*}
However, under our scheme, we can rewrite:
\[\e\left[\sum_{(x, h)\in \tilde{\landsc}_2} \mathds{1}_{(0, 0) \text{ is eventually relayed by } (x,  h)}\right] = \e\left[\sum_{(x, h)\in \landsc} \mathds{1}_{(0, 0) \text{ is eventually relayed by } (x,  h)}\right] = \e[T].\]
Furthermore
\[\e\left[\sum_{(x, h)\in \tilde{\landsc}_1} \mathds{1}_{h = 0, (x, h) \text{ is eventually relayed by } (0,  H)}\right] = \e[\# \{x\in \Phi: (x, 0)\text{ is eventually relayed by } (0, H) \text{ in }\tilde{\landsc}\}].\]

Which finally leads to 
\[\lambda' \frac{\lambda}{\lambda+\lambda'} \e[\ell(\zone(0, H))] = \frac{\lambda'}{\lambda+\lambda'} \e[T].\]
Which is exactly the announced result.
\end{proof}

\begin{remark}
Here, the adaptation of the scheme to zero height buildings is not necessary. A way to make it work otherwise is to extend marks to a product space with a second component indicating whether the point is a building or a user. This allows one to extend the result to both the case with regular $0$ height building, and to get the span of the connection zone at any given height, at the cost of heavier notation. This of course leads to the following corollary:
\end{remark}
 \begin{corollary}\textit{Area of the typical relaying zone.}\\
 For $h\geq 0$, let $T_h$ be the hitting time of the cemetery state $(-\infty, -\infty, -\infty)$ for the Markov chain $(\tilde{x}_n, \mathrm{t}_n, H_n)$ started at 
$(0, h, +\infty)$. The expected area of the zone of eventual relaying $\zone (0, H)$ is equal to 
\[\int_{h\geq 0}\frac{1}{\lambda}\e[T_h]dh.\]
 \end{corollary}

\subsection{General Range}

As we saw, under a finite horizontal range, we were able to recover some computational results, mainly due the similar definition of blockage. Indeed, a crucial property we assumed was the fact that signal can only be blocked by buildings the tops of which are visible. In the following, this is not necessarily the case (see Figure \ref{block_chiant}). As such, while our methods still apply, they fail to provide as tractable results. 

\begin{definition}\textbf{General Range.}

We call \textit{range} a convex compact subset $\mathrm{R}$ of $\R^ 2$ that is symmetric ($x\in R \leftrightarrow (-x)\in \R$). For $x_1, x_2$ two points of $\Half$, we say that $x_2$ of  is in range from $x_1$ if
\[x_2 \in \mathrm{R}+x_1.\]
\end{definition}

In this context, one needs to be more careful when defining the blocking building: indeed, buildings that are not in range can still block the signal, leading to a loss of line of sight connection:

\begin{figure}[h]
            \centering
            \includegraphics[width=0.75\textwidth]{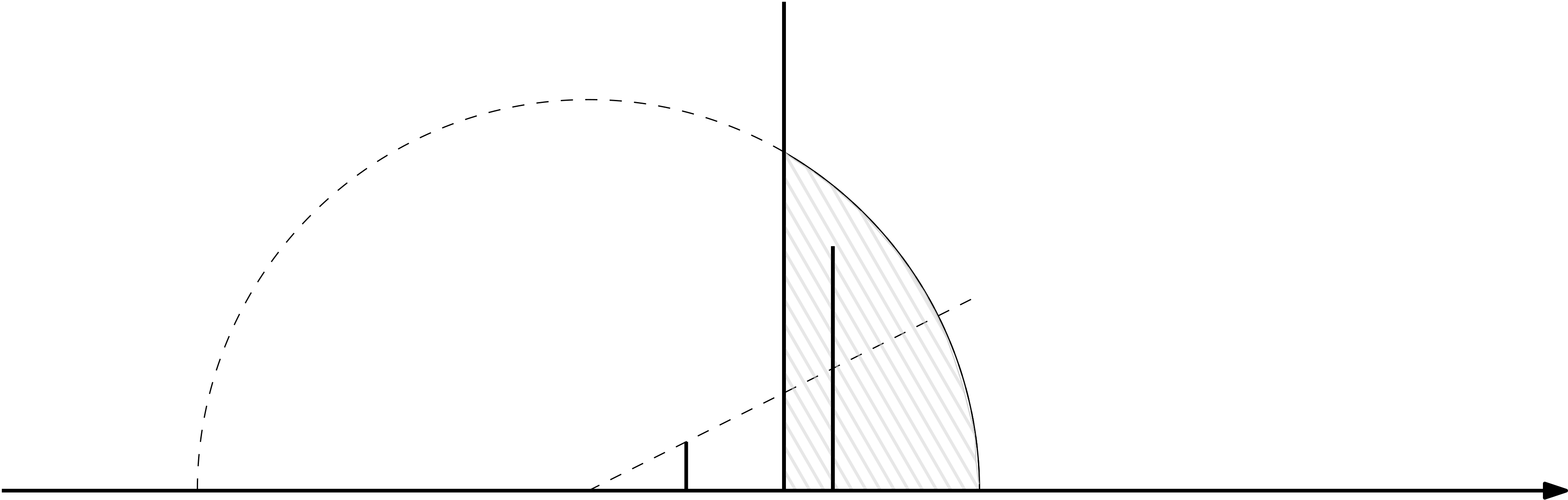}
            \caption{\label{block_chiant}In the case of Euclidian balls, the presence of tall buildings outside of the ball can prevent relaying further right. Here, the blockage angle is lower than the angle between the origin and a building in range due to this phenomenon.}
\end{figure}

\begin{definition}\textbf{Visibility, relaying scheme.}

Let $\landsc$ be a landscape.

Let $(a_1, b_1)$ and $(a_2, b_2)$ be two buildings of $\landsc$. We say that $(a_2, b_2)$ is visible from $(a_1, b_1)$ if it is in range and there are no buildings above the segment from $(a_1, b_1)$ to $(a_2, b_2)$. In other words, if the following conjunction of events holds true:
\[\{(a_2, b_2)\text{ is in range from } (a_1, b_1)\} \cap \below{[a_1, a_2]}{(a_1, b_1)}{\frac{b_2-b_1}{a_2-a_1}}.\]

Thanks to this, we can define the relaying scheme $\tau$ as
\[\tau_{\mathrm{R}} (x, y, \landsc) = (x_b, y_b) := \underset{\substack{(a, b) \in \landsc\cap \mathrm{R}+(x, y) \\(a, b) \text{ visible from }(x, y)}}{\text{argmax}} \frac{b-y}{a-x}.\]
When this is well defined, and $\tau (x, y, \landsc) = (-\infty, -\infty)$ otherwise.

The blockage angle and its tangent are defined as usual in this context
\end{definition}

Due to this condition of visibility, the computation distribution of the blocking building needs more care.
In order to account for this complication, let us first give an alternative characterization of the blocking building of a given point of $\Half$:

\begin{definition}\textbf{Stoppage Point.}

Let $\landsc$ be a given landscape.
Let $(x, y)\in\Half$ be a point. We define its stoppage point as the $x-coordinate$ starting from which line of sight from $(x, y)$ is cut off for all altitudes.
\[X_{stop}(x, y) := (\min\{a\in [x, +\infty[: \exists (a, b)\in \landsc, \forall (a, c)\in \mathrm{R}+(x, y), c<b\})\wedge (\max\{a : \exists (a, b)\in \mathrm{R}+(x, y)\}).\]
We can also consider the stoppage \textit{height} $X_{stop}(x, y)$ defined as either the height of the building inducing stoppage when it exists or $-\infty$ when it does not.
\end{definition}

In a similar way to what we noticed for the reverse-blockage in Proposition \ref{length}, there are two possibilities for the position of the stoppage point: either it is the base of a tall building which top is out of range, blocking further transmission, or it is simply the $x-$coordinate of the rightmost point in range. This result in the distribution of the stoppage point having two parts: one with density with respect to Lebesgue measure and the other having positive mass in one point. It can be seen through the following lemma:

\begin{lemma}\textbf{Distribution of the Stoppage Point.}\label{dist_stop}

For $a\in \R$, let us note the $h_a(\mathrm{R})$ highest point in $\mathrm{R}$ above $a$:
\[h_{\mathrm{R}}(a):= \sup \{b: (a, b)\in \mathrm{R}\}.\]
In the last definition we can for instance adapt the convention that the supremum of an empty set is minus infinity.

Let $\landsc$ be a given random landscape.\\
Then, for $(x, y) \in \Half$ the stoppage point $X_{stop}(x, y)$ of $(x, y)$ in $\landsc$ has its CDF given by the following reformulation of events:
\[\{X_{stop}(x, y) \geq t\} =  \mathrm{below}([x, t],h_{\mathrm{R}+(x, y)})\cap \{t\leq\max\{a : \exists (a, b)\in \mathrm{R}+(x, y)\}\} .\]

When $\landsc$ is a Poisson Building Process of intensity $\lambda$ and height distribution $\mathcal{L}$, Lemma \ref{dessous_diago} yields the following:
\[\p (X_{stop}(x, y) \geq t) = \exp \left(\int_{x}^{t} -\lambda \mathcal{L}(]h_{\mathrm{R}+(x, y)}(a), \infty[) \rm{d}a \right)\mathds{1}_{t \leq \max\{a : \exists (a, b)\in \mathrm{R}+(x, y)\}} .\]
In other words, $X_{stop}(x, y)$ has its law equal to
\[g\times\Leb_{|[x,\max\{a : \exists (a, b)\in \mathrm{R}+(x, y)\}]}+ \ \exp \left(-\lambda\bigintsss_{\hspace{-10pt}x}^{\max\{a : \exists (a, b)\in \mathrm{R}+(x, y)\}} \hspace{-90pt} \mathcal{L}(]h_{\mathrm{R}+(x, y)}(a), \infty[) \rm{d}a \right)\delta_{\max\{a : \exists (a, b)\in \mathrm{R}+(x, y)\}},\]
where 
\[g: u\mapsto \lambda \mathcal{L}([h_{\mathrm{R}+(x, y)}(u), +\infty[)\exp\left(-\lambda \int_{x}^{u}\mathcal{L}(h_{\mathrm{R}+(x, y)}(a), +\infty[ \rm{d}a\right).\]

Moreover, conditioned on the stoppage point, the stoppage height is distributed according to $\mathcal{L}$ conditioned on being greater that $h_{\mathrm{R}+(x, y)}(X_{stop}(x, y))$.
\end{lemma}

In the specific case of Poisson building processes, this knowledge of the law of the stoppage point is useful: similarly as in Corollary \ref{typical}, the spatial Markov property ensures that the buildings in $x+\mathrm{R}\cap [x, X_{stop}(x, y)]\times \R_+$ form a Poisson point process of intensity $\lambda\times \mathcal{L}$.

\begin{remark}
In the following, we will use the more general formulation for building processes that are inhomogeneous $2$D Poisson Point processes of intensity $\lambda(x)\rm{d}x\otimes \mathcal{L}(x)$, which is a direct consequence of the lemma. More precisely, it will be useful in the case of Poisson building processes conditioned to have the tops in a given area of the half plane $\Half$ :
\[\p (X_{stop}(x, y) \geq t) = \exp \left(\int_{x}^{t} -\lambda(a) \mathcal{L}(a)(]h_{\mathrm{R}+(x, y)}(a), \infty[) \rm{d}a \right)\mathds{1}_{t \leq \max\{a : \exists (a, b)\in \mathrm{R}+(x, y)\}} .\]
\end{remark}

Notice that one can reformulate the blocking building of a given point $(x, y)$ as follows:
\[\tau_{\mathrm{R}} (x, y, \landsc) = (x_b, y_b) := \underset{\substack{(a, b) \in \landsc\cap \mathrm{R}+(x, y) \\a<X_{stop}(x, y)}}{\text{argmax}} \frac{b-y}{a-x}.\]
Thanks to this observation, in the case of Poisson building processes, we are now able to give computational results on the distribution of blockage in finite range.

\begin{proposition}\textbf{Distribution of the blocking building: general finite range.}\label{evo_finite_gen}

For $x_1< x_2\in \R$, $h_1, h_2\in \R_+$, and $t_1 = \frac{h_2-h_1}{x_2-x_1}$ let us denote
\begin{align*}
    \placehold((x_1, h_1), (x_2, h_2)) : =\bigintsss_{\hspace{-5pt}x_2}^{\max\{a : \exists (a, b)\in \mathrm{R}+(x_1, h_1)\}}\hspace{-100pt}\lambda \mathcal{L}([h_{\mathrm{R}+(x_1, h_1)}&(u), +\infty[)\\
&\exp\left(-\lambda \int_{x_2}^{u}\mathcal{L}([(h_2+(a-x_2)t_1)\wedge h_{\mathrm{R}+(x_1, h_1)}(a), +\infty[ \rm{d}a\right)\rm{d}u \\
+ \exp\Biggl(-\lambda&\bigintsss_{\hspace{-5pt}x_2}^{\max\{a : \exists (a, b)\in \mathrm{R}+(x_1, h_1)\}}\hspace{-100pt} \mathcal{L}([(h_2+(a-x_2)t_1)\wedge h_{\mathrm{R}+(x_1, h_1)}(a), +\infty[\rm{d}a\Biggl).
\end{align*}
Then for $(X, H)\in \Half$ a given point, the restriction of the law of $(x_b, y_b)$ to $\Half$, or in other words, the restriction of the law of the blocking building to the event of its existence has density with respect to the product measure $\Leb \otimes \mathcal{L}$, which is given by
\begin{align*}
g_{\mathrm{R}}: (x, h)_{0\leq n\leq N} \mapsto \lambda &\exp\Biggl(-\lambda\frac{(\primi(h)-\primi(H))(x-X)}{h-H}\Biggl) \placehold((X, H), (x, h)) \mathds{1}_{(x, h)\in \mathrm{R}+(X, H)},
\end{align*}
where, if $h=H$, we consider the continuous extension $\frac{(\primi(h)-\primi(H))}{h-H} = (1-F(h))$.
\end{proposition}

\begin{proof}
This is simply a reformulation of the definition of blockage:

The event
\[\tau_{\mathrm{R}}(X, H) = (x, h)\leftrightarrow (x, h) = \underset{\substack{(a, b) \in \landsc\cap \mathrm{R}+(x, y) \\a<X_{stop}(x, y)}}{\text{argmax}} \frac{b-y}{a-x}\]
is equivalent to the conjunction of the following events:
\begin{enumerate}
    \item $(x, h)$ is in $\landsc\cap\mathrm{R}+(X, H)$.
    \item The event $\below{[X, x[}{(X, H)}{\frac{h-H}{x-X}}$.
    \item The stoppage point $X_{stop}(X, H)$ of $(X, H)$ is after $x$.
    \item The event $\below{]x, X_{stop}(X, H)[}{(x, h)}{\frac{h-H}{x-X}}$.
\end{enumerate}
Then, one simply has to use the spatial Markov property in the order of these events:
\begin{enumerate}
    \item The presence of the building yields the indicator as well as the density part.
    \item $\below{[X, x[}{(X, H)}{\frac{h-H}{x-X}}$ yields the part $exp\Biggl(-\lambda\frac{(\primi(h)-\primi(H))(x-X)}{h-H}\Biggl)$.
    \item Conditional on the previous event, $\landsc$ is a Poisson point process of intensity $\lambda\Leb \otimes \mathcal{L}$ in $\Half\cap \R_{> x_2}\times \R_+$, independent from the first two events. Moreover, the stoppage point cannot be before $x$ (from convexity of the range $\mathrm{R}$). As such, we can use Lemma \ref{dist_stop} to find its position.
    \item Conditioned on the stoppage point, the buildings between $x$ and $X_{stop}(X, H)$ form a Poisson point process of intensity $\lambda\Leb \otimes \mathcal{L}$ in $\mathrm{R}+(X, H)\cap ]x, X_{stop}(X, H)[\times \R_+$ independent from the two last steps. This, together with the distribution of $X_{stop}(X, H)$, yield the part $\placehold((X, H), (x, h))$.
\end{enumerate}
\end{proof}

Now, while the distribution of the blocking building is well understood, we have no simple formula for the joint evolution as we had through Proposition \ref{evo_finite} or Theorem \ref{evo}. This, as the absence of direct Markovian dynamics, can be seen as a consequence of the absence of decreasing characteristic of the shade. 

As such, while in the case of stationary processes, all the results of Subsection \ref{ev_co_fini} still apply, we lack the computational results to leverage them.
\section*{Conclusion}

In this paper, we generalized the model of \cite{Junse_Francois}, giving precise computational results as well as structural properties for the evolution of visibility from a user, and for the number of connections to a given relay. Through this study, we highlighted two key properties: On the one hand, the framework of \textit{Poisson building processes}, thanks to the spatial Markov property, allows us to derive detailed computational results, thanks to the absence of correlation \textemdash \ and independently of the stationarity, as our methods straightforwardly adapt to inhomogeneous point processes, up to keeping in memory the position of the blocking building if one wants to keep the Markovian characteristic in Corollary \ref{Markov}. On the other hand, the stationarity of the process of buildings and the equivariant characteristic of our connection scheme allows us to use the mass transport principle, tying our relaying schemes to the theory of \textit{Eternal Family Trees}, immediately yielding structural properties and allowing us to have access to estimates of the backwards dynamic through the knowledge of the forward one, such as in Proposition \ref{mass_transport_save}.

When looking at both properties at the same time, that is our framework of \textit{homogeneous Poisson Building process}, we end up with a computationally well understood Eternal Family Tree, which we called the \textit{Relay EFT}. Leveraging this computational framework, further studies of this object for its belonging to the class of unimodular point shifts might be envisioned, for instance a study of its fractal dimension (as defined in \cite{unimod_dim}) depending on the height distribution $\mathcal{L}$, though such studies fall outside of the scope of present article.

Lastly, we showed the limits of our methods in the general range case, where we can only recover partial information. Further studies could focus on a higher dimensional case, though the finite range behaviour is expected to greatly vary, as does the behaviour of continuum percolation.
\newpage
\printbibliography
\end{document}